\documentclass[final]{siamart171218}

\usepackage{lipsum}
\usepackage{amsfonts,amssymb}
\usepackage{graphicx}
\usepackage{epstopdf}
\usepackage{todonotes} 
\usepackage{booktabs}
\usepackage{braket}
\usepackage{bm}
\usepackage{bbm}
\usepackage{algorithmic}
\ifpdf
  \DeclareGraphicsExtensions{.eps,.pdf,.png,.jpg}
\else
  \DeclareGraphicsExtensions{.eps}
\fi

\newcommand{\TheTitle}{Local theory for spatio-temporal canards \\ and delayed bifurcations}
\newcommand{\TheShortTitle}{Spatio-temporal canards and delayed bifurcations}
\newcommand{\TheAuthors}{D. Avitabile, M. Desroches, R. Veltz, M. Wechselberger}

\headers{\TheShortTitle}{\TheAuthors}

\title{{\TheTitle}}

\author{
  Daniele Avitabile
  \thanks{Department of Mathematics, Vrije Universiteit Amsterdam
   and MathNeuro Team, Inria Sophia Antipolis (\email{d.avitabile@vu.nl}).}
  \and
  Mathieu Desroches
  \thanks{MathNeuro Team, Inria Sophia Antipolis.}
  \and
  Romain Veltz
  \footnotemark[2]
  \and
  Martin Wechselberger
  \thanks{School of Mathematics and Statistics, The University of Sydney.}
}

\usepackage{amsopn}
\usepackage{mathtools}

\newcommand{\ind}{\operatorname{\mathbbm{1}}}

\newcommand{\iunit}{i}
\renewcommand{\phi}{\varphi}

\DeclareMathOperator{\spec}{\sigma}
\DeclareMathOperator{\range}{range}
\DeclareMathOperator{\id}{id}
\DeclareMathOperator{\real}{Re}
\DeclareMathOperator{\imag}{Im}
\DeclareMathOperator{\e}{e}
\DeclareMathOperator{\diag}{diag}
\DeclareMathOperator{\spn}{span}
\DeclareMathOperator{\meas}{meas}

\newcommand{\RSet}{\mathbb{R}}

\newcommand{\NSet}{\mathbb{N}}
\newcommand{\ZSet}{\mathbb{Z}}
\newcommand{\SSet}{\mathbb{S}}

\newcommand{\CSet}{\mathbb{C}}
\newcommand\blank{{\mkern 2mu\cdot\mkern 2mu}}
\newcommand{\eps}{\varepsilon}

\newcommand{\cE}{E}

\newcommand{\cM}{M}

\newcommand{\cV}{V}
\newcommand{\cX}{X}  
\newcommand{\cY}{Y}  
\newcommand{\cZ}{Z}

\newcommand{\norm}[1]{\left \| #1 \right \|}

\newsiamremark{hypothesis}{Hypothesis}
\crefname{hypothesis}{Hypothesis}{Hypotheses}
\newsiamremark{problem}{Problem}
\crefname{problem}{Problem}{Problems}
\newsiamremark{remark}{Remark}
\crefname{remark}{Remark}{Remarks}

\usepackage{comment}

\newcommand{\vecd}[2]{\left(#1, #2\right)}

\graphicspath{{Figures/}}

\ifpdf
\hypersetup{pdftitle={\TheTitle},pdfauthor={\TheAuthors}}
\fi

\begin{document}
\maketitle

\begin{abstract}
  We present a rigorous framework for the local analysis of canards and slow passages
  through bifurcations in a wide class of infinite-dimensional dynamical systems
  with time-scale separation. The framework is applicable to models where an
  infinite-dimensional dynamical system for the fast variables is coupled to a
  finite-dimensional dynamical system for slow variables. We prove the existence of
  centre-manifolds for generic models of this type, and study the reduced,
  finite-dimensional dynamics near bifurcations of (possibly) patterned steady
  states in the layer problem. Theoretical results are complemented with detailed
  examples and numerical simulations covering systems of local- and
  nonlocal-reaction diffusion equations, neural field models, and delay-differential
  equations. We provide analytical foundations for numerical observations recently
  reported in literature, such as spatio-temporal canards and slow-passages through
  Hopf bifurcations in spatially-extended systems subject to slow parameter
  variations. We also provide a theoretical analysis of slow passage through a Turing
  bifurcation in local and nonlocal models.
\end{abstract}

\section{Introduction}\label{sec:intro}

Many physical and biological systems consist of processes that evolve on disparate
time- and/or length-scales and the observed dynamics in such systems reflect these
multiple-scale features as well. Mathematical models of such multiple-scale systems
are considered singular perturbation problems with two-scale problems as the most
prominent.

In the context of ordinary differential equations (ODEs),  singular perturbation
problems are usually discussed under the assumption that there exists a coordinate
system such that observed slow and fast dynamics are represented by corresponding
slow and fast variables globally, i.e., the system of ODEs under consideration is
given in the \emph{standard (fast) form} 
\begin{equation}\label{eq:slowFastFin}
\begin{aligned}
  \dot{u}& = F(u,v,\mu,\eps)\\
  \dot{v}& = \eps G(u,v,\mu,\eps),
\end{aligned}
\end{equation}
where $(u,v)\in\RSet^n\times\RSet^m$, $F$ and $G$ are smooth functions on
$\RSet^n\times\RSet^m\times\RSet^p\times\RSet_{>0}$, and $0<\eps\ll 1$ is the timescale separation
parameter. 
The {\em geometric singular perturbation theory (GSPT)} to analyse such finite
dimensional singular perturbation problems  is well established. It was pioneered by
Neil Fenichel in the 1970s \cite{Fenichel:1979} and is based on the notion of {\em
normal hyperbolicity}\footnote{Precise definitions for this and other GSPT concepts
can be found in the Appendix.} which refers to a spectral property of the equilibrium
set $S$ of
the \emph{layer problem} (also known as the \emph{fast subsystem}) obtained in the
limit $\eps\to 0$ in \eqref{eq:slowFastFin}.
In GSPT, this set $S:=\{(u,v) \in \RSet^n \times \RSet^m \colon F(u,v,\mu,0)=0\}$ is
known as the \emph{critical manifold}
since it is assumed to be an $m$-dimensional differentiable manifold. Normal
hyperbolicity refers then to the property that the (point) spectrum of the critical
manifold is bounded away from the imaginary axis (on every compact subset of $S$). 

\emph{Loss of normal hyperbolicity} is a key feature in finite dimensional
singular perturbation problems for rhythm generation as observed, for example, in the
famous van der Pol relaxation oscillator. An important question in this context is
how the transition from an excitable to a relaxation oscillatory state occurs in
such a system. The answer is (partially) given by the \emph{canard phenomenon} which
was discovered by French mathematicians who studied the van der Pol relaxation
oscillator with constant forcing \cite{BCDD1981}. They showed an explosive growth
of limit cycles from small Hopf-like to relaxation-type cycles in an exponentially
small interval of the system parameter. This parameter-sensitive behaviour is hard
to observe, and the corresponding solutions in phase space resemble (with a good
portion of imagination) the shape of a `duck' which explains the origin of the
nomenclature. Note, this phenomenon is degenerate since it only exists in
one-parameter families of slow-fast vector fields in $\RSet^2$. Furthermore, it is
always associated with a nearby \emph{singular} Hopf bifurcation, i.e. frequency and
amplitude of the Hopf cycles depend on the singular perturbation parameter $\eps\ll
1$. A seminal work on van der Pol canards~\cite{BCDD1981} was obtained
by Eric Beno{\^i}t, Jean-Louis Callot, Francine and Marc Diener through
\textit{non-standard analysis} methods and soon after similar results were reached
using standard matched asymptotics techniques by Eckhaus~\cite{E1983}.
A decade later, further results were obtained using \textit{geometric
desingularization} or \textit{blow-up} by Dumortier and Roussarie~\cite{DR96}, and
soon after by Krupa and Szmolyan~\cite{KS01a,KS01c,KS01b}.

Fortunately, this degenerate situation does not occur in systems with two (or more)
slow variables where canards are generic, i.e.~their existence is insensitive to
small parameter perturbations. Beno\^{i}t \cite{Ben83} was the first to study generic
canards in $\RSet^3$.  He also observed how a certain class of generic canards (known
as canards of {\em folded node} type) cause unexpected rotational properties of
nearby solutions. Extending geometric singular perturbation theory to canard problems
in $\RSet^3$, Szmolyan and Wechselberger \cite{SW01} provided a detailed geometric
study of generic canards.
In particular, Wechselberger \cite{W05} then showed that rotational properties of
folded node type canards are related to a complex  local geometry of invariant
manifolds near these canards and associated bifurcations of these canards.
Coupling this local canard structure with a global return mechanism can explain
complex oscillatory patterns known as mixed-mode oscillations (MMOs); see, e.g.,
\cite{W05,BKW06,Desroches2012}. This is closely related to canards of folded node and
{\em folded saddle-node} type. Firing threshold manifolds or, more broadly, transient
separatrices in slowly modulated excitable systems form another important class of
applications \cite{WALC11,WMR13} which is a different type of canard mechanism which
includes \textit{folded saddle} canards.

A necessary (but not sufficient) condition for canards is the crossing of a real
eigenvalue in the point spectrum of $S$ which leads geometrically  to a folded
critical manifold. Another important singular perturbation phenomenon is the {\em
delayed loss of stability through a Hopf bifurcation
}\cite{neishtadt87,neishtadt88,HKSW16} associated with the crossing of a complex
conjugate pair of eigenvalues in the spectrum of $S$. We note that this phenomenon
refers to a Hopf bifurcation in the layer problem which is, in general,  different to
a singular Hopf bifurcation in the full system \eqref{eq:slowFastFin} with $0<\eps\ll
1$, as the two bifurcations are not necessarily related. A prime application of this
phenomenon is associated with \emph{elliptic bursters} in neurons, see
for instance \cite{IZH07}.

Canard theory has been extended to arbitrary finite dimensions in \cite{W12}, 
i.e.~dimensional restrictions on the slow variable subspace are not necessary to study
these problems. Slow-fast theory for ODEs plays also an important role in the construction
of heterogeneous (and possibly relative) equilibria in PDEs posed on the real line.
This construction relies on a
spatial-dynamic ODE formulation \cite{sandstede2002stability}: stationary states and
travelling waves are identified with homo- or hetero-clinic connections of an ODE in
the spatial variable, which may exhibit spatial-scale separation. These
\textit{global orbits} may display canard segments, and can be constructed using GSPT
theory for ODEs
\cite{doelman1997pattern,buv2006canard,marszalek2012fold,harley2014existence,%
carter2016stability,avitabile2017ducks,carter2017transonic}.

A further interesting class of intermediate problems between ODEs and the ones
considered here arise when the spatial-dynamical system obtained in travelling waves
problems is itself infinite-dimensional, or does not have an obvious phase space. In
this context, a GSPT may not be readily available, and recent contributions in this
area have been provided by Hupkes and Sandstede for lattice equations
\cite{hupkes2010traveling,hupkes2013stability}, and by Faye and Scheel for nonlocal
equations with convolutions~\cite{faye2018center}.

More generally, the literature on infinite-dimensional slow-fast
dynamical systems is less developed than its ODE counterpart. The papers by Bates and
Jones~\cite{Bates:1989jn} and by Bates, Lu, and Zeng~\cite{Bates:1998iw},
include historical background and references to the state-of-the-art
literature (as of the end of 1980s and 1990s, respectively) on invariant manifolds
for infinite-dimensional systems, including in
particular~\cite{Chow:1988co,Vanderbauwhede:1992ft,Bates:1994ix,Wayne:1997ho,VanMinh:2004kz,Henry:2006i}.
In addition, Bates, Lou and Zeng~\cite{Bates:1998iw} provided a seminal contribution
for the persistence of invariant manifolds for problems posed on Banach spaces. The
theory presented therein is general, albeit the derivation of GSPT for
infinite-dimensional equations is problem-dependent (see, for instance, the one
derived by Menon and Haller for the Maxwell-Bloch equations~\cite{Menon:2001if}).

The construction of orbits displaying canard segments or delayed bifurcations in the
infinite-dimensional setting remains an open problem, with two important obstacles:
(i) the loss of normal hyperbolicity, and (ii) the connection of slow and fast orbit
segments with the view of obtaining a global, possibly periodic, orbit. The present
paper addresses the
former, and provides a general local theory for infinite-dimensional problems of the
following type
\begin{equation}\label{eq:slowFastInf-2}
\begin{aligned}
  \dot{u}& =  Lu + R(u,v,\mu,\eps) := F(u,v,\mu,\eps)\\
  \dot{v}& = \eps G(u,v,\mu,\eps)
\end{aligned}
\end{equation}
where the fast variables $u$ belong to a Banach space $X$, and the slow variables $v$
to $\RSet^m$, for some $m \in \NSet$. The vector $\mu\in\RSet^p$, $p \in \NSet$,
refers to a set of (possible) aditional control parameters.\footnote{Precise
definitions of the infinite dimensional dynamical system under study will be given in
\cref{sec:generalSetup}.} We say that \cref{eq:slowFastInf} is a system of
$m$-slow, \textit{$\infty$-fast differential equations}. 

As we discuss below, systems of the form \cref{eq:slowFastInf-2} are prevalent in the
limited literature currently available on canards and delayed bifurcations in
infinite-dimensional dynamical systems. The first contribution to
this topic was given in the early 1990s by Su~\cite{Su199438}, who studied a system
of type \cref{eq:slowFastInf-2} in which the fast variables $u=(u_1,u_2)$ evolve
according to
the FitzHugh--Nagumo model, with diffusivity in the voltage $u_1$, and diffusionless
recovery variable $u_2$, subject to a heterogeneous, slowly-increasing current with
amplitude $v$. For this setup, Su proved the existence of orbits displaying a delayed
passage through a Hopf bifurcation. 
Two decades later, de Maesschalck, Kaper, and Popovic \cite{deMaesschalck2009canards}
proved existence of trajectories modelled around the canard phenomenon in the
nonlinear example
\[
  \varepsilon \partial_t u = \varepsilon^\mu \partial_{xx} u + V(u,x,t,\eps) u,
  \qquad \mu >0,
\]
using the specific scaling in conjunction with the method of lower and upper solutions.
An intuition of this group was to bypass the difficulties associated to the loss of
normal hyperbolicity, by hard-wiring a slow, non-monotonic evolution in the function
$V$, as opposed to prescribing a coupled dynamic for a slow variable $v$, as we do here. 

In 2015, Tzou, Ward and Kolokolnikov \cite{tzou2015slowly} studied slow passages
through a Hopf bifurcation in a reaction-diffusion PDE with a slowly-varying
parameter, using asymptotic methods which Bre\~na-Medina, Avitabile, Champneys and
Ward concurrently adopted to investigate slow-fast orbits connecting patterned, metastable
states in plant dynamics \cite{brena2015stripe,avitabile2018spot}. In the same year
Chen, Kolokolnikov, Tzou and Gai gave numerical evidence of slow passages through
Turing bifurcations in an advection--reaction--diffusion equation with slowly-varying
paraters subject to noise, with applications in vegetation
patterns~\cite{chen2015patterned}.

In 2016, Krupa and Touboul studied canards in a delay-differential equation (DDE)
with slowly-varying parameters. This work differs from the other mentioned in this
literature review, because it does not deal with a spatially-extended system, albeit
the fast variable $u$ lives in a Banach space.

In 2017, Avitabile, Desroches and Knobloch introduced canards in spatially extended
systems (termed \textit{spatio-temporal canards}), producing numerical evidence of
cycles and transients containing canard segments of folded-saddle and folded-node
type. They studied a $2$-slow $\infty$-fast system coupled to a  neural-field
equation posed on the real line and on the unit sphere. Analytical predictions for
spatio-temporal canards were derived in a regime where interfacial dynamics provides
a dimensionality reduction of the problem, amenable to standard GSPT analysis.

More recently, in 2018, two papers obtained results similar to the ones in
Su~\cite{Su199438} but on different models, and proposing different techniques: Bilinsky
and Baer~\cite{Bilinsky2018} used a WKB-expansion to study spatially-dependent
buffer points (buffer curves) in heterogeneous reaction-diffusion equations;
Kaper and Vo~\cite{kaper2018delayed} found numerical evidence of delayed-Hopf
bifurcations in several examples of heterogeneously, slowly-driven reaction-diffusion
systems, and give a formal, accurate, asymptotic argument to compute the associated
buffer curves.\\

\noindent

The main contribution of this article is the derivation of a local theory
for canards and slow passages through bifurcations in $\infty$-fast $m$-slow
systems: not only is this a natural and necessary step for the construction of global
orbits, but it is an achievable target for generic models. A key ingredient of our
framework is a centre-manifold reduction of \cref{eq:slowFastInf-2}, 
which is an infinite-dimensional system. A general centre-manifold theory is
available (see a recent literature review by Roberts \cite{roberts2019}), but has not
been used for infinite-dimensional problems of type~\cref{eq:slowFastFin} to overcome
the loss of normal hyperbolicity. We proceed
systematically, as follows:
\begin{enumerate}
  \item The layer problem $\dot u = F(u,v,\mu,0)$ is a differential equation in $X$,
    with parameters in $\RSet^{m+p}$. We assume that this system admits a bifurcation
    of a steady state $u_*$ at $(v,\mu) = (v_*,\mu_*)$. 
    Thus the linear operator $D_uF(u_*,v_*,\mu_*,0)$ has a non-empty centre spectrum
    and normal hyperbolicity fails. Note that $u_*$ is constant in time, but need
    not be constant in space ($u_*$ is be a heterogeneous steady state of the layer
    problem).
  
  \item Under the assumption that $D_uF$ has a spectral gap, and its
    centre spectrum has finite dimension $n \in \NSet$, we prove the existence of a
    $(\mu,\eps)$-dependent centre manifold \emph{of the original problem} \cref{eq:slowFastInf-2}.
  \item We perform a centre-manifold reduction: in a
    neighbourhood of $(u_*,v_*,\mu_*,0) \in \cX \times \RSet^{m+p+1}$, the $m$-slow
    $\infty$-fast system reduces to an $m$-slow $n$-fast one. The latter inherits the
    slow-fast structure of the former, because our reduction acts trivially on the
    slow variables $v$.
  \item We derive a normal form for the reduced system, which still exhibits loss of
    normal hyperbolicity at the origin. Existing ODE results can however be used to
    prove the local existence of canards and slow passages through bifurcations.
\end{enumerate}

Centre-manifold theory is at the core of the procedure described above, and our
treatment relies on existing tools on this topic. In particular, we adopt the
notation and formalism developed by Vanderbauwhede and
Iooss~\cite{Vanderbauwhede:1992ft}, and expanded in
great detail in a book by Haragus and Iooss~\cite{Haragus2010}. 
We build our results around the ones presented in the latter, to which we refer for
further reading.

The theory presented here justifies rigorously numerical evidence presented in recent
literature on the subject, and provides new results: we do not assume weak
diffusivity in the linear operators, nor locality of the linear/nonlinear operators,
and we provide theory and numerical examples for generic systems of
integro-differential equations, local- and nonlocal-reaction diffusion problems,
and DDEs. The centre-manifold reduction can be carried out
for generic singularities, and this enables us to present theory for delayed-passages
through Turing bifurcations, which had not been studied analytically before, and are
specific to partial differential equations (PDEs).

We aim to present content in a format that is hopefully useful to readers working in
both finite- and infinite-dimensional slow-fast problems, and this impacts on the
material and style of the paper. For instance, centre-manifold reductions
for infinite-dimensional systems rely on checking a series of technical
assumptions on the layer problem, hence we took some measures to make the material
self-contained, and to template our procedure for a relatively large class of
problems. We believe that several
applications should be covered by our treatment, or can be derived with minor
modifications. For these reasons we structured
the paper as follows: in \cref{sec:numericalExamples} we provide a suite of
numerical examples that show canard phenomena and related slow passage through Hopf
and Turing bifurcations in PDEs, DDEs and integro-differential equations; in
\cref{sec:generalSetup} we introduce the notation and functional-analytic setup used
in the following sections; \cref{sec:centreManifoldReduction} exposes steps 1--3 of the
procedure described above, for generic $m$-slow, $\infty$-fast systems;
\cref{sec:reductions} describes step 4 in the procedure for fold, Hopf, and Turing
bifurcations in generic systems; in \cref{sec:applications} we go through several
cycles of steps 1--4, showing how they can be used in the concrete
applications presented in \cref{sec:numericalExamples}; we conclude in
\cref{sec:conclusions}.

\section{Numerical examples}
\label{sec:numericalExamples}
We provide numerical examples of system of $\infty$-fast, $m$-slow variables to which
our analytical framework is applicable. In this section we discuss evidence obtained
for slow passage through saddle-node, Hopf, and Turing bifurcations in various
models. All computations involve numerical bifurcation analysis of steady states or
periodic orbits of the layer problem, and time-stepping of the full problem. All
computations in spatially extended systems are performed using a MATLAB suite for
generic problems developed in \cite{rankin2014continuation} (see
\cite{daniele_avitabile_2020_3821169} for a recent tutorial). Computations for the DDE are
performed using DDE-BIFTOOL \cite{engelborghs2001dde}.

\subsection{Folded saddle and folded node canards in a neural field model}
\begin{figure}
  \centering
  \includegraphics{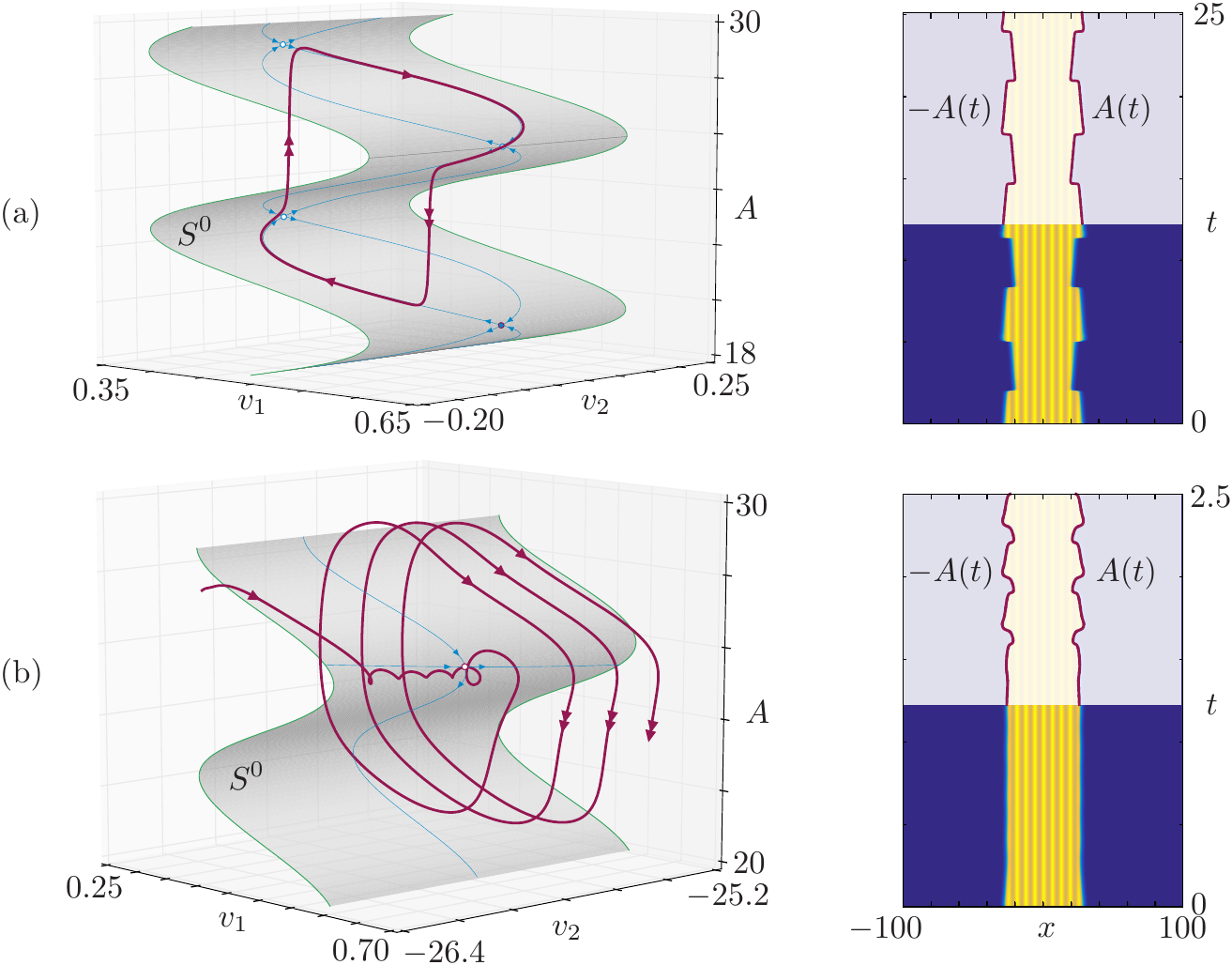}
  \caption{Examples of spatio-temporal canards in \cref{eq:NF}. (a) Periodic orbit
  containing a folded-saddle canard segment obtained with $\theta =
  \theta_s$, $w(x,y) = \kappa_1 \exp(-|x-y|)(\kappa_2 + \kappa_3 \cos(y/\kappa_4))$,
  $\Omega = \RSet/2l\ZSet$, $l=100$, $a =
0.5$, $b=c=0$, $\kappa_1 = 0.5$, $\kappa_2=1$, $\kappa_3=0.5$, $\kappa_4=1$, $\mu =50$. Left: orbit
in the $(A,v_1,v_2)$-space, with $A$ defined in \cref{eq:levelSet}, superimposed on
the critical manifold $S_0$ of the fast subsystem associated to \cref{eq:NF} in the
variables $(A,v_1,v_2)$. Right: corresponding spatio-temporal solution, in which the
level sets are shown for illustrative purposes. (b) Examples of an orbit containing a
folded-node canard segment, obtained from the example in (a) by setting $a=c=1$,
$b=0$. Adapted from \cite{avitabile2017spatiotemporal}.}
  \label{fig:NeuralFieldRing}
\end{figure}
We begin with an integro-differential equation from mathematical neuroscience, namely
a neural field model posed on a compact domain $\Omega \subset
\RSet^d$,
\begin{equation}\label{eq:NF}
  \begin{aligned}
    \partial_t u & = -u + \int_\Omega w(\blank,y) \theta(u(y,t), v_1 ) \,d\rho(y)  
		 && \textrm{in $\Omega \times \RSet_{>0}$,}\\
  \dot v_1 & = \eps \bigg( v_2 + c \int_\Omega \theta(u(y,t),v_1) \, d\rho(y) \bigg) 
	       && \textrm{in $\RSet_{>0}$,}\\
  \dot v_2 & = \eps \bigg(-v_1 + a + b \int_\Omega \theta(u(y,t),v_1)\, d\rho(y)
                            \bigg)  	      
	   && \textrm{in $\RSet_{>0}$,}
  \end{aligned}
\end{equation}
where $\theta$ model a firing rate function and $w$ the synaptic connections (see
\cite{coombes2014neural} for a recent review on neural fields). The firing rate
function is commonly modeled via a sigmoidal or a Heaviside function
\[
  \theta_s(u,h) = \frac{1}{1+\exp(-\mu (u-h))}, \quad
  \theta_{H}(u,h) = H(u-h).
\]
In contrast to standard neural field models, the firing-rate threshold $v_1$
oscillates slowly and harmonically in time if $b=c=0$, and its evolution is coupled
to the fast neural field activity variable $u$ if $b$ or $c$ are nonzero.
\begin{figure}
  \centering
  \includegraphics{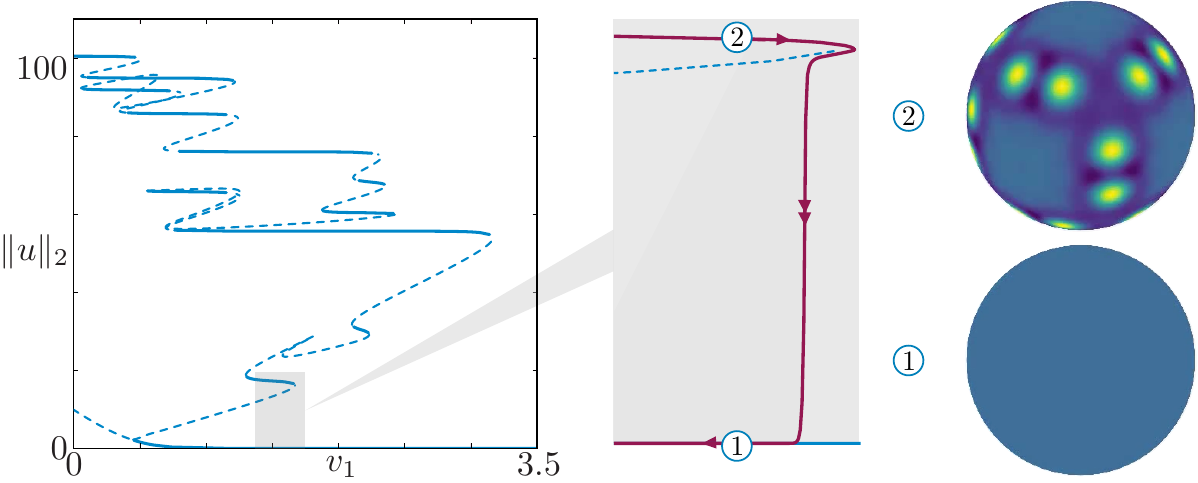}
  \caption{Example of solutions containing spatio-temporal canards of folded-saddle
  type for the neural field \cref{eq:NF} posed on $\Omega = \SSet^2$. Left
  bifurcation diagram of the fast subsystem. Centre: trajectory of the model with
  $\eps \ll 1$, near a fold of the fast subsystem. Right: exemplary patterns. We
refer the reader to \cite{avitabile2017spatiotemporal}, from where this figure is
adapted, for details about parameters and computations.}
  \label{fig:NeuralFieldSphere}
\end{figure}

In \cite{avitabile2017spatiotemporal}, spatio-temporal canards were introduced and
analysed using interfacial dynamics, valid in the case of Heaviside firing rate
$\theta_{H}$ posed on one-dimensional domains. If $\Omega \subset \RSet$ and $\theta
= \theta_{H}$, one can construct even solutions $u(x,t) = u(-x,t)$ from of their
$v_1$-level set on $\RSet_{\geq 0}$,
\begin{equation}\label{eq:levelSet}
  A(t) = \{ x \in \RSet_{\geq 0} \colon u(x,t) = v_1(t) \}.
\end{equation}
If the $v_1$-level set has a single connected component, then $A(t)$ is a scalar function, and an
exact evolution equation in the
variable $(A,v_1,v_2) \in \RSet^3$ can be derived for the model and studied using
GSPT for ODEs. Solutions $(A(t),v_1(t),v_2(t))$,
with canard segments correspond spatio-temporal canards of the original model.

In the present paper we study the case $\theta = \theta_s$, on generic
domains $\Omega$ for which we
report numerical simulations adapted from \cite{avitabile2017spatiotemporal} in
\cref{fig:NeuralFieldRing,fig:NeuralFieldSphere}: spatio-temporal canards of
folded-saddle and
folded-node type are predicted from the theory, are found numerically with a
sigmoidal firing rate (\cref{fig:NeuralFieldRing}), and they also persist in
higher spatial dimensions, where the theory does not apply
(\cref{fig:NeuralFieldSphere}). In the present paper we will introduce a rigorous
treatment of this problem, valid for generic firing rates, kernels, and domains. 

The
functional setting for this problem will be the Banach space of continuous,
real-valued functions defined on $\Omega$. The theory developed in the following
sections
is also
applicable to the Swift--Hohemberg equation subject to slow parameter variation, as
studied in \cite{gandhi2015new}, albeit a natural functional setting for this
problem requires Hilbert spaces. We shall give below several examples for problems on
Hilbert spaces.

\subsection{Slow passage through a Hopf bifurcations in PDEs and DDEs} 
\begin{figure}
  \centering
  \includegraphics{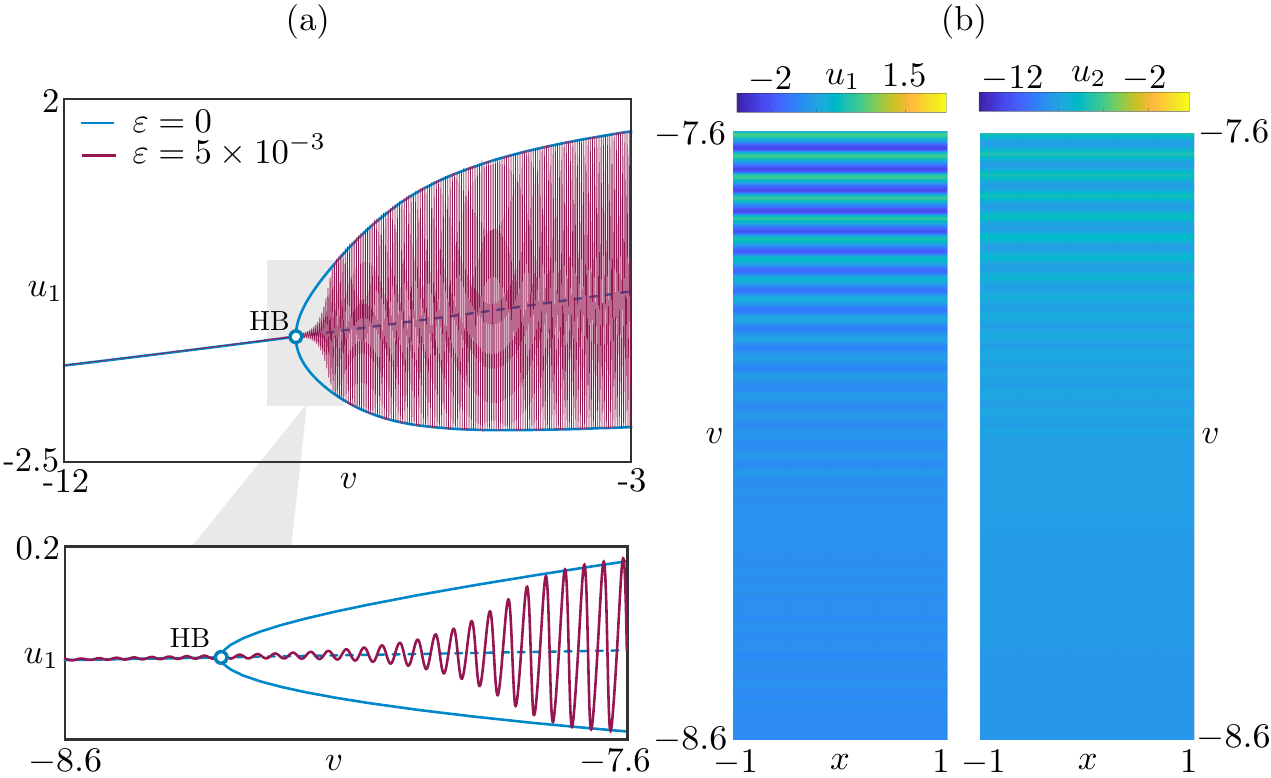}
  \caption{Slow passage through a Hopf bifurcation in the Fitzhugh--Nagumo model
  \cref{eq:FN}. Parameters: $d_1=0.1$, $b=0.1$, $c=0.05$, $\eps = 0, 5 \times 10^{-3}$.}
  \label{fig:FHNSlowPassageHopf}
\end{figure}
A second class of examples pertains slow passages through Hopf bifurcations in
infinite-dimensional systems. A first numerical example is given by the following
reaction-diffusion PDE with slowly-varying parameters:
\begin{equation}\label{eq:FN}
  \begin{aligned}
    \partial_t u_1 & = d_1\partial_x^2 u_1 + u_1 -u_1^3/3 -u_2 + v
		   && \text{ in $(0,2\pi) \times \RSet_{>0}$,} \\
    \partial_t u_2 & =  u_1 + c - b u_2                       
		   && \text{ in $(0,2\pi) \times \RSet_{>0}$,} \\
            \dot v & = \eps                                     
		   && \text{ in $\RSet_{>0}$,}
  \end{aligned}
\end{equation}
subject to periodic boundary conditions
\[
  u_1(0,t) = u_1(2\pi,t), \qquad u_2(0,t) = u_2(2\pi,t),
  \qquad t \in \RSet_{>0},
\]
where $v$ plays the role of an external current. In this case we will show that the
fast subsystem admits a homogeneous steady state undergoing a Hopf bifurcation, and
will apply our theory to conclude that system \cref{eq:FN} displays slow passage
through a Hopf bifurcation, of which we give numerical evidence in
\cref{fig:FHNSlowPassageHopf}.  This example shows a slow-passage through a
bifurcation of a \emph{homogeneous} steady state, because the boundary conditions
allow it. We highlight that the theory presented below is equally
applicable to slow-passages through \emph{heterogeneous} (patterned) steady states,
of which the neural field problem is an example.

The analytical treatment of this problem will be done on Sobolev spaces, and the
presentation of this example assumes no prior knowledge of Centre Manifold reduction
for PDEs, but only basic notions in functional analysis. The theory developed in the
next sections is applicable to generic reaction--diffusion PDEs, and in particular
to the examples considered in \cite{kaper2018delayed}. 

A separate analytical treatment is instead required for the slow-passage through
a Hopf bifurcation in DDEs, for which we consider the model problem
\begin{equation}\label{eq:DDE}
\begin{aligned}
  \dot{u}(t)         & = v(t) u(t) - u^3(t) - u(t-\tau)+d 
		     && \textrm{in $\RSet_{>0}$}\,,\\
\dot{v}(t) & = \eps
		     && \textrm{in $\RSet_{>0}$}\,.
\end{aligned}
\end{equation}
The fast subsystem of~\cref{eq:DDE} is therefore the one-component delay-differential
equation, in which the trivial equilibrium undergoes a first Hopf bifurcation at
$v=-3/4$. The slow
drift on $\alpha$ induces the slow passage presented in \cref{fig:DDE}.

\begin{figure}
  \centering
  \includegraphics{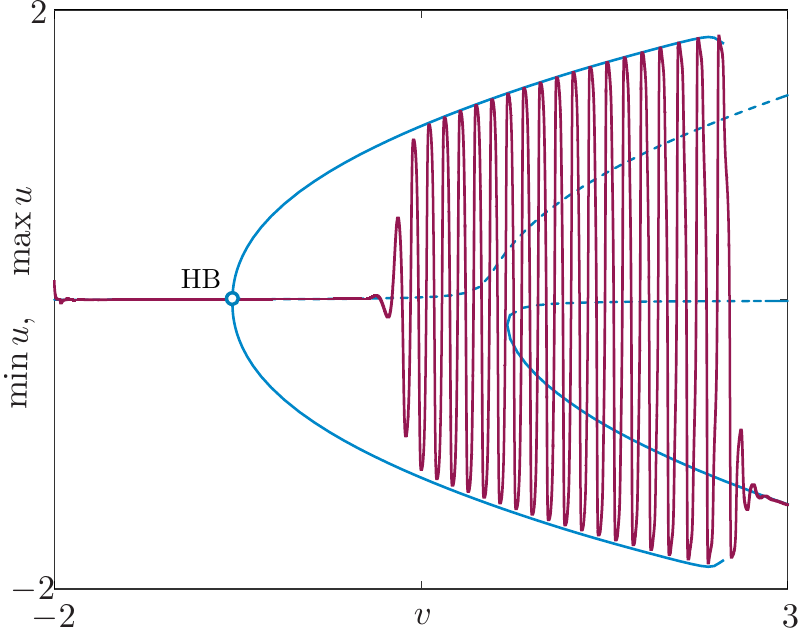}
  \caption{Slow passage through a Hopf bifurcation in the delay-differential
    equation \cref{eq:DDE}.
  Parameter values: $\eps=0.01$, $\tau=4$, $d=0.01$.}
	\label{fig:DDE}
\end{figure}

\subsection{Slow passage through Turing bifurcations in local- and nonlocal-reaction
diffusion model}
Since centre-manifold reductions are possible for generic bifurcations in PDE or
integro-differential equations, we
consider the case of a slow passage through a Turing bifurcation. In general, we will study problems of the
following type
\begin{equation}
  \begin{aligned}
    \partial_ t u & = D \partial_x^2 u + f(u,v) 
		  && \textrm{in $\Omega \times \RSet_{>0}$}, \\
	   \dot v & = \eps
		  && \textrm{in $\RSet_{>0}$,} 
  \end{aligned}
\end{equation}
subject to suitable boundary conditions, where $u$ denotes the concentrations of $q$
reactants, and $D= \diag\{d_1,\ldots,d_q\}$ with $d_i>0$ for all $i \in \NSet_q$, is
a diffusion matrix. While developing a theory for generic models of this type is
possible, we concentrate on two main examples, which can easily be generalised.

\subsubsection{Nonlocal-reaction diffusion model} \label{sec:numericsNonlocalTuring}
The first example under consideration is a model with nonlocal reaction, which has
been studied for $\eps=0$ in the context of population dynamics
\cite{Genieys:2006ji,Volpert:2011wt}. 
\begin{equation}\label{eq:nonlocalRD}
  \begin{aligned}
    \partial_t u & = d \partial_x^2 u + (v-b) u - u \int_\Omega w(\blank - y)
    u(y,t)\,dy, \qquad  
		 && \textrm{in $\Omega \times \RSet_{>0}$}\\
      \dot v & = \eps
		 && \textrm{in $\RSet_{>0}$}
  \end{aligned}
\end{equation}
where $\Omega = \RSet/2l\ZSet$, and $w$ is an interaction kernel. The model undergoes a Turing-like bifurcation of
a homogeneous steady state, and we show numerical results of a slow-passage through
this bifurcation (see
\cref{fig:nolocalRDSlowPassage}). We will derive theoretical results for this
$1$-component, nonlocal problem, and present a strategy that is easily adaptable to generic $d$-component
reaction diffusion equation subject to Dirichlet, Neumann, or mixed homogeneous
boundary conditions, under mild hypotheses on $f$. This strategy can be used also in
reaction-diffusion problems undergoing other bifurcations.
\begin{figure}
  \centering
  \includegraphics{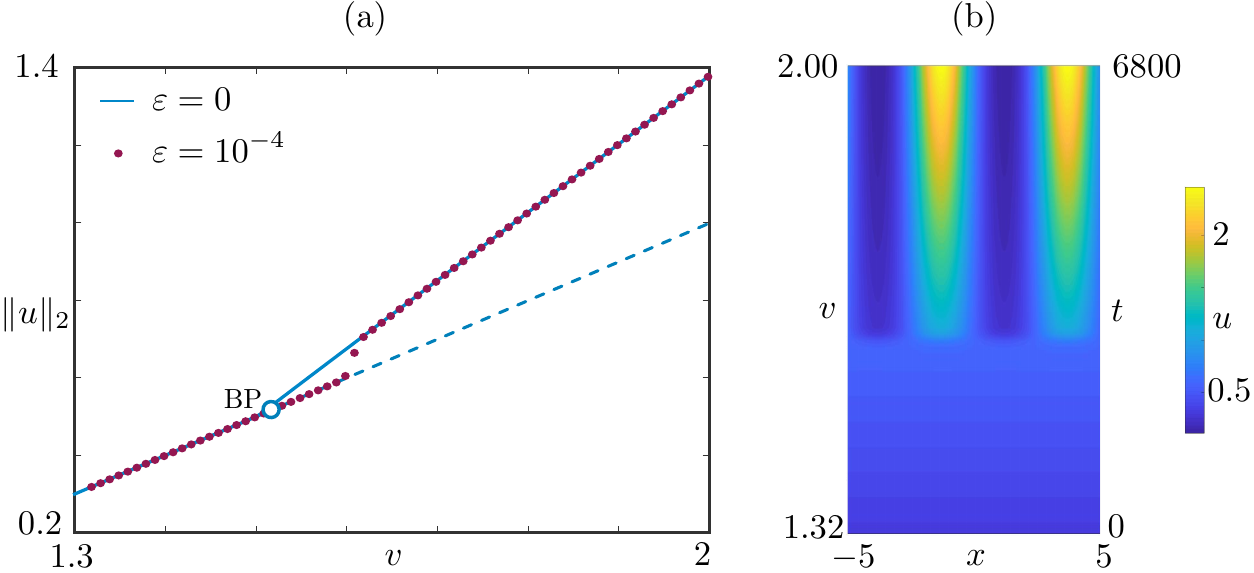}
  \caption{Slow passage through a Turing-like bifurcation for the nonlocal
    reaction-diffusion equation \cref{eq:nonlocalRD} (see
    \cite{Genieys:2006ji,Volpert:2011wt} for details). (a): Bifurcation diagram of
    the fast subsystem, where $\eps = 0$ and $v$ is a parameter (blue); the
    homogeneous steady state undergoes a Turing-like bifurcation at $v\approx 1.5$.
    A time simulation with slowly-varying $v$, where $\Vert u \Vert_2$ changes
    slowly, displays a slow passage through the bifurcation (red). (b) Spatiotemporal
    profile of the time simulation. On the vertical
    axis we report both $t$ and $v$, which increases linearly. Computations are
    performed with interaction function $w(x) = (2h)^{-1} \ind_{(-h,h)}(x)$ and
  parameters $h = 3$, $d=0.05$, $b=1$, $l=5$ and $\eps=0, 10^{-4}$.}
  \label{fig:nolocalRDSlowPassage}
\end{figure}

\subsubsection{Schnakenberg model} \label{sec:numericsSchnakenberg}
\begin{figure}
  \centering
  \includegraphics{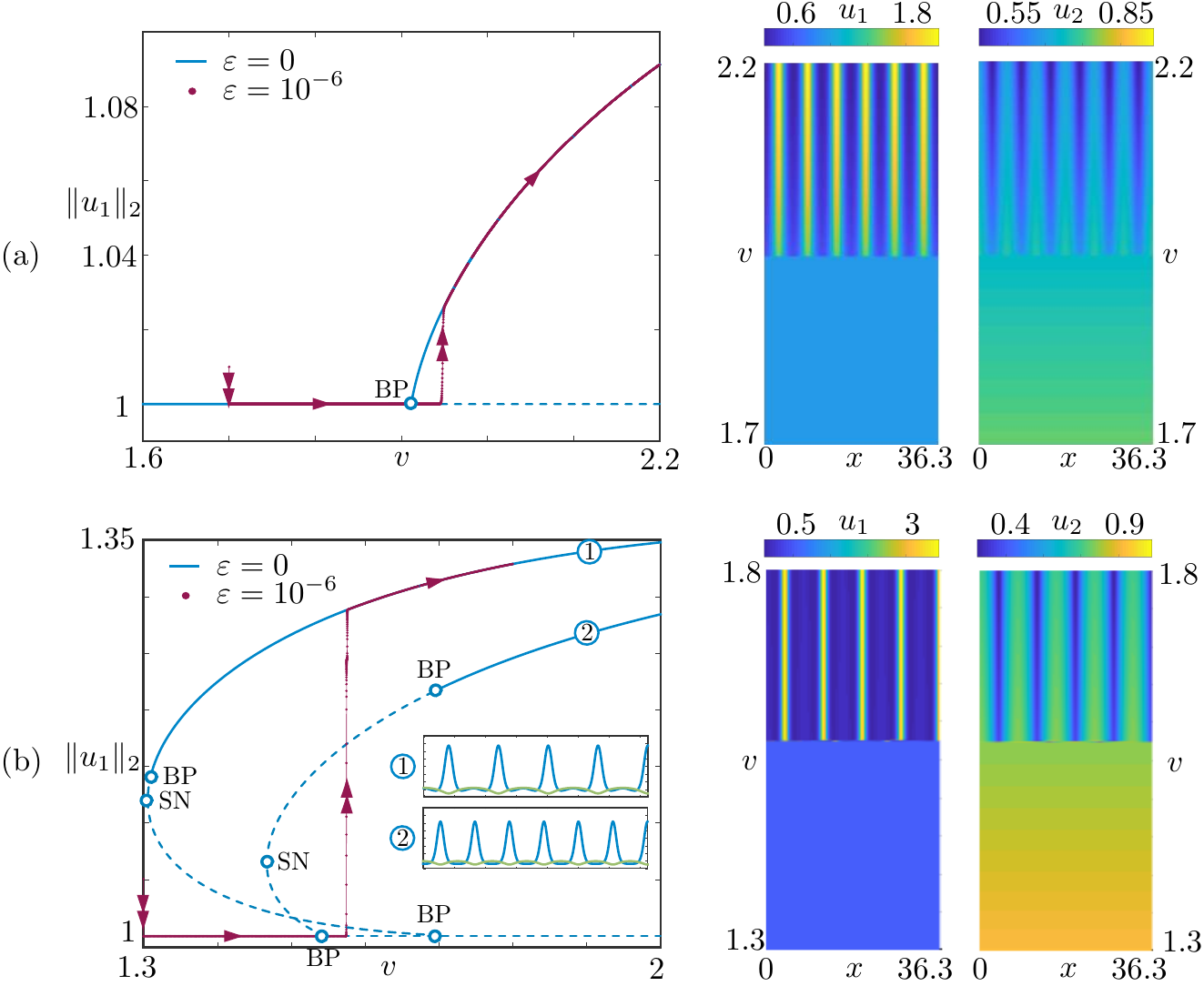}
  \caption{Slow passage through a supercritical (a) and subcritical (b) Turing
  bifurcation of the Schnakenberg model \cref{eq:Schnakenberg}. In the subcritical
setting, the slow passage is through a first Turing bifurcation of the homogeneous
steady state, albeit the solution jumps to a stable branch emanating from a second
Turing bifurcation, which restabilises after a fold and a secondary bifurcation.
Parameters: $c=r=h=b=1$, $d_2=10$, and (a) $d_1=0.26$, (b) $d_1=0.1497$.}
\label{fig:slowPassageSchnakenberg}
\end{figure}
As an example of a local-reaction diffusion equation, we consider the following
Schnakenberg model, which undergoes a sub- or
super-critical Turing bifurcation, as discussed in~\cite{BrenaMedina:2014},
\begin{equation}\label{eq:Schnakenberg}
  \begin{aligned}
    \partial_t u_1 & = d_1 \partial_x^2 u_1 + vu_1^2 u_2 - (c+r)u_1 + hu_2, & x \in (0,l) \\
    \partial_t u_2 & = d_2 \partial_x^2 u_2 - vu_1^2 u_2 + cu_1 - h u_2 + b, & x \in (0,l) \\
            \dot v & = \eps,                                     
  \end{aligned}
\end{equation}
subject to Neumann boundary conditions
\[
  \begin{aligned}
  & \partial_x u_1(0,t) = \partial_x u_1(l,t) = 0, \\
  & \partial_x u_1(0,t) = \partial_x u_1(l,t) = 0.
  \end{aligned}
\]
Numerical results for this model are reported in \cref{fig:slowPassageSchnakenberg}.

\section{General setup for infinite dimensional problems}

This section provides the specific setup for $\infty$-fast, $m$-slow systems
necessary such that centre manifold theory
from~\cite{Haragus2010,Vanderbauwhede:1992ft} can be adopted.

\label{sec:generalSetup}
\subsection{Notation} 
We will use a unique symbol, $\vert \blank \vert$, to denote the 2-norm in $\RSet^n$,
for any integer $n$.
We indicate by $\cX$, $\cY$, $\cZ$ Banach spaces on the field $\RSet$ or $\CSet$,
endowed with norms $\Vert \blank \Vert_{\cX}$, $\Vert \blank \Vert_{\cY}$,
and $\Vert \blank \Vert_{\cZ}$, respectively. Throughout this article, we will assume the
continuous embeddings $\cZ \hookrightarrow \cY \hookrightarrow \cX$. 

For a given Banach space $\cX$, we denote by $C_\eta(\RSet,\cX)$ the
following weighted space:
\begin{equation}
  C_\eta(\RSet,\cX) = \big\{ u \in C^0(\RSet,\cX) \colon \Vert u \Vert_{C_\eta(\RSet,\cX)}= \sup_{t \in \RSet}
  \e^{-\eta |t|} \Vert u(t) \Vert_\cX  < \infty \big\}.
\end{equation}
which is a Banach space equipped with the norm $\Vert \blank \Vert_{C_\eta(\RSet,\cX)}$.

Furthermore, we indicate by $\mathcal{L}(\cZ,\cX)$ the set of bounded linear operators
from $\cZ$ to $\cX$ and by $\cZ^*$, $\cX^*$ the dual spaces of $\cZ$ and
$\cX$, respectively. 
The adjoint of a linear operator $L \in \mathcal{L}(\cZ,\cX)$ is the operator $L^*
\in \mathcal{L}(\cX^*,\cZ^*)$ satisfying
\begin{equation}\label{eq:adjointDef}
  \langle f, Lz \rangle_{\cX} = \langle L^* f, z \rangle_{\cZ}
  \qquad
  \textrm{for all $z \in \cZ$ and for all $f \in \cX^*$,}
\end{equation}
where $\langle \blank, \blank \rangle_{X} \colon \cX^* \times \cX \to \RSet$ is the
canonical duality pairing between $X$ and its dual $X^*$ (a similar definition holds
for $\langle \blank, \blank \rangle_{\cZ}$). 

The resolvent set of a linear operator $L:Z \subset X \to X$ is the set of complex
numbers 
$
  \rho(L) = \{ 
   \lambda \in \CSet \colon \textrm{$(\lambda\id-L)$ is invertible and $(\lambda\id-L)^{-1}\in\mathcal L(X)$}
 \},
 $
and the spectrum of $L$ is the complement of the resolvent set 
$
  \sigma(L)=\mathbb C\setminus\rho(L).
$
A complex number $\lambda$ is called an eigenvalue of $L$ if
$\ker(\lambda\id-L)\neq\{0\}$ and the non null elements of this kernel are called the
eigenvectors of $L$.

The $k$th Fr\'echet derivative of a nonlinear operator $R \in C^k(\cZ,\cY)$ will be
denoted by $D^k R$. For the Taylor expansions of $R \in C^k(\cZ,\cY)$ around $0 \in
\cZ$ we will use the compact notation
\[
  R(u) = \sum_{n=0}^k R_n(u^{(n)}) + o(\Vert u \Vert^k_{\cZ})
\]
where $u^{(n)} = (u,\ldots,u) \in (\cZ)^n$ represents $u \in \cZ$ repeated $n$
times and 
the $n$-linear operator $R_n \in \mathcal{L}^n(\cZ,\cX) = \mathcal{L}(\prod_{j=1}^n \cZ,\cX)\cong
\mathcal{L}(\cZ,\mathcal{L}(\cZ,\ldots,\mathcal{L}(\cZ,\cX)))$, with
$n$ recursions \cite[Lemma 1.6]{Chow:2012tv}, acts as
\[
  R_n(u^{(n)}) = \frac{1}{n!} D^n R(0)u^{(n)}.
\]

\subsection{Systems of $m$-slow, $\infty$-fast differential equations}
For ease of reference, we recall system \eqref{eq:slowFastInf-2}, 
i.e.~we fix $m, p \in \NSet$, and study
differential equations posed on $\cX \times \RSet^m$, featuring $p+1$
control parameters, of the following type:
\begin{equation}\label{eq:slowFastInf}
\begin{aligned}
  \dot{u}& =  Lu + R(u,v,\mu,\eps) := F(u,v,\mu,\eps)\\
  \dot{v}& = \eps G(u,v,\mu,\eps)
\end{aligned}
\end{equation}
where $L$ is a linear operator while $R$, $F$, and $G$ are nonlinear operators. More
precisely we assume $L \colon \cZ \to \cX$, $R \colon \cZ \times
\RSet^m \times \RSet^p \times \RSet \to \cY$, $F \colon \cZ \times \RSet^m \times
\RSet^p \times \RSet \to \cX$, and $G \colon \cZ \times \RSet^m \times \RSet^p \times
\RSet \to \RSet^m$. We have collected the $p+1$ control parameters in the vector
$(\mu,\eps) \in \RSet^p \times \RSet$. This choice reflects the special role played
by the parameter $\eps$, which is taken to be small and determines a separation
of time scales between the \emph{fast} variables $u$ and the \emph{slow} variables
$v$. We are interested in applications where $X$ is infinite dimensional, therefore
we say that \cref{eq:slowFastInf} is a system of $m$-slow, $\infty$-fast differential
equations. We note that, in general, owing to the splitting $F(u) = L u +
R(u)$ one poses $Z \subset Y \subset X$~\cite{Vanderbauwhede:1992ft,Haragus2010} and,
depending on the problem, we may
have $Y=X$ or $Y\neq X$. As we shall mention below, the choice $Y
\neq X$, when possible, simplifies the verification of certain hypotheses.
%

It is sometimes advantageous to write the system \cref{eq:slowFastInf} in a more
compact form. To this end, we introduce the spaces
$\tilde \cX = \cX \times \RSet^m$, $\tilde \cY = \cY \times \RSet^m$ and
$\tilde \cZ = \cZ \times \RSet^m$, endowed with the corresponding natural norms and
embeddings, and rewrite \cref{eq:slowFastInf} as a differential equation on $\tilde
\cX$
\begin{equation}\label{eq:slowFastInfComp}
  \frac{d}{dt}\tilde u  = \tilde L \tilde u + \tilde R (\tilde u,\mu,\eps) := \tilde
  F(\tilde u,\mu,\eps).
\end{equation}
where $\tilde u = (u,v)$ and
\[
  \tilde L =
  \begin{pmatrix}
    L & D_{v}R(0,0,0,0) \\
    0      & 0
  \end{pmatrix},
  \qquad
  \tilde R( \tilde u,\mu,\eps) = 
  \begin{pmatrix}
    R(u,v,\mu,\eps) - D_vR(0,0,0,0)v\\
    \eps G(u,v,\mu,\eps).
  \end{pmatrix}
\]

\section{Centre manifold reduction} \label{sec:centreManifoldReduction}
We aim to perform a centre manifold reduction of
\cref{eq:slowFastInf} near bifurcations points of equilibria of the fast subsystem
\begin{equation}\label{eq:fastSubsystemInf}
  \dot u = F(u,v,\mu,0) = Lu + R(u,v,\mu,0)
\end{equation}
in which $(v,\mu) \in \RSet^m \times \RSet^p$ is fixed. As we shall see below, this
in turn will provide a reduction of \cref{eq:slowFastInf} to a set of $m$-slow,
$n$-fast ODEs, which can subsequently be studied with standard tools for slow-fast
ODEs. Solutions to \cref{eq:fastSubsystemInf} are defined below.
\begin{definition}[Solution]\label{def:solution}
  A solution of the differential equation \eqref{eq:fastSubsystemInf} is a function
  $u:I\to\cZ\hookrightarrow\cX$ defined on an interval $I\subset\RSet$ such that $u\in
  C(I,\cZ)\cap C^1(I,\cX)$, and equality \eqref{eq:fastSubsystemInf} holds in $\cX$
  for all $t\in I$.	
\end{definition}

\subsection{Preliminary hypotheses}
In order to perform the reduction of the fast subsystem, we employ the framework
introduced by Haragus and Iooss
in~\cite{Vanderbauwhede_center:1987,Vanderbauwhede:1992ft,Haragus2010}
We report below a set of hypotheses that are necessary to perform the centre manifold
reduction. These hypotheses are adapted (or reproduced) from
\cite{Haragus2010}.
We begin with an assumption on the vector field \cref{eq:slowFastInf}, and on a
subset of its equilibria.
\begin{hypothesis}[Vector field]
  \label{hyp:vectField}
  The operators $L$, $R$ in \cref{eq:slowFastInf} satisfy the
  following conditions:
  \begin{romannum}
    \item $L \in \mathcal{L}(\cZ,\cX)$.
    \item There exist an integer $k \geq 2$ and neighbourhoods $\cV_u \subset \cX$,
      $\cV_v \subset \RSet^m$, $\cV_\mu \subset \RSet^p$, $\cV_\eps \subset \RSet$ of
      $0$ such that $R \in C^k(\cV_u \times \cV_v \times \cV_p \times
      \cV_\eps,\cY)$ and
      \[
	R(0,0,0,0) = 0 \qquad D_uR(0,0,0,0) = 0
      \]
  \end{romannum}
\end{hypothesis}

\Cref{hyp:vectField} implies that $0 \in \cX \times \RSet^m$ is an equilibrium of
\cref{eq:slowFastInfComp} for $(\mu ,\eps) = 0$, that is, $0 \in \cX$ is an
equilibrium of the fast subsystem~\cref{eq:fastSubsystemInf} for
$(v,\mu) = 0$. Note that \Cref{hyp:vectField}.(ii) implies $F(0,0,0,0) = 0$, while in
general $G(0,0,0,0) \neq 0$.

\begin{hypothesis}[Spectral decomposition] \label{hyp:spectralDecomposition}
  The spectrum $\sigma(L)$ of $L$ can be written as follows
  \[
    \spec(L)= \sigma_u \cup \sigma_c \cup \sigma_s,
  \]
  where
  \[
    \sigma_u := \{ \lambda \in \spec(L) \colon \real \lambda  > 0 \}, \quad
    \sigma_c := \{ \lambda \in \spec(L) \colon \real \lambda  = 0 \}, \quad
    \sigma_s := \{ \lambda \in \spec(L) \colon \real \lambda  < 0 \},
  \]
  satisfy the following assumptions
  \begin{romannum}
  \item The set $\sigma_c$ consists of a finite number $n>0$ of
    eigenvalues with finite algebraic multiplicities.
    \item There exists a positive constant $\gamma >0 $ such that
      \[
	\inf_{\lambda \in \sigma_u} \real \lambda > \gamma, \qquad
	\sup_{\lambda \in \sigma_s} \real \lambda < -\gamma
      \]
  \end{romannum}
\end{hypothesis}
If \cref{hyp:vectField,hyp:spectralDecomposition} hold, we can introduce the spectral
projector on the centre subspace, $P_c \in \mathcal{L}(\cX,\cZ)$
associated with $\sigma_c$ and its complementary projector $P_{su} = \id_\cX - P_c
\in \mathcal{L}(\cX) \cap \mathcal{L}(\cY) \cap \mathcal{L}(\cZ)$. We therefore
have the decomposition
\[
  \cX = \cE_c \oplus X_{su}, \quad 
  \cE_c = \range P_c = \ker P_{su} \subset \cZ, \quad
  \cX_{su} = \range P_{su} = \ker P_c \subset \cX,
\]
we set $Z_{su} = P_{su} \cZ \subset \cZ$ and $Y_{su} = P_{su} \cY \subset \cY$, and
we let $L_c \in \mathcal{L}(E_c)$, $L_{su} \in \mathcal{L}(Z_{su},Y_{su})$, be the
restrictions of $L$ to $E_c$ and $Z_{su}$, respectively. We are now ready to state
one final hypothesis necessary for the centre manifold reduction:

\begin{hypothesis}[Linear equation in the hyperbolic subspace]
    \label{hyp:linearEquation}
    For any $\eta \in [0,\gamma]$, where $\gamma$ is
    the spectral gap defined in \cref{hyp:spectralDecomposition}, and any $f \in
  C_\eta( \RSet, \cY_{su})$, the linear problem
  \begin{equation}\label{eq:linProbFastSubs}
    \dot{u}_{su} = L_{su} u_{su} + f(t),
  \end{equation}
  has a unique solution $u_{su} = K_{su} f \in C_\eta(\RSet,\cZ_{su})$. The linear map
  $K_{su}$ belongs to $\mathcal{L}\big( C_\eta( \RSet,\cY_{su}), C_\eta( \RSet,
  \cZ_{su}) \big)$, and there exists $\Pi \in C([0,\gamma], \RSet)$
  such that
  \begin{equation}\label{eq:PhiBoundFast}
    \Vert K_{su} \Vert_{
      \mathcal{L}( C_\eta( \RSet, \cY_{su}), C_\eta( \RSet, \cZ_{su}))}
      \leq \Pi(\eta).
  \end{equation}
\end{hypothesis}

Checking this hypothesis in applications may be nontrivial, as it requires to prove
that there is a gain in regularity of solution $u(t)\in\cZ$ to
\cref{eq:linProbFastSubs}, when the forcing term $f(t)$ belongs to $\cY$. Note that
this forcing term mimics the effect of the nonlinearity and, as such, shares it
range. In some cases, this hypothesis can be substituted by an estimate on the
resolvent set of $L$, and in the examples treated below we will give pointers to the
literature for these cases. In particular, when $Y \neq X$, a necessary condition to
check the validity of \cref{hyp:linearEquation} is available~\cite[Section
2.2.3]{Haragus2010}. In this paper we focus on the more challenging case $Y=X$, and
present examples where a necessary condition is available, and examples where it is
not (hence \cref{hyp:linearEquation} must be checked directly).

\subsection{Centre-manifold results}
We now proceed to derive a centre manifold existence result for \cref{eq:slowFastInf}
in a neighbourhood of $0 \in \cX \times \RSet^m  \times \RSet^p
\times \RSet$.

\begin{theorem}[Parameter-dependent centre manifold] 
  \label{thm:centreManifoldWithParameters}
  Assume
  \crefrange{hyp:vectField}{hyp:linearEquation}. Further, assume $G \in C^k(V_u
    \times V_v \times V_p \times
  V_\eps, \RSet^m)$. Then the $m$-slow, $\infty$-fast system \cref{eq:slowFastInf}
  admits a finite-dimensional centre manifold. More
  precisely, there exists a map 
  $\Psi \in C^k(E_c \times \RSet^m \times \RSet^q \times \RSet,Z_{su})$ with
  \[
    \Psi(0,0,0,0) = 0, \quad D_u \Psi(0,0,0,0) = 0, 
  \]
  and a neighbourhood $O_u \times O_v \times O_\mu \times O_\eps$ of $0 \in \cZ
  \times \RSet^m \times \RSet^q \times \RSet$ such that, for all $(\mu,\eps) \in
  O_\mu \times O_\eps$ the manifold
  \begin{equation}\label{eq:CMDefinition}
  \cM_{c}(\mu,\eps) = \big\{ (u_c,v) + \big( \Psi(u_c,v,\mu,\eps), 0 \big) \colon 
      u_c \in E_c,\; v \in O_v \big\}
  \end{equation}
  has the following properties:
  \begin{romannum}
  \item $\cM_{c}(\mu,\eps)$ is locally invariant, that is, if $(u,v)$ is a solution of
    \cref{eq:slowFastInf} such that $(u(0),v(0)) \in \cM_{c}(\mu,\eps) \cap (O_u \times
    O_v)$ for all $t \in [0,T]$, then $(u(t),v(t)) \in \cM_{c}(\mu,\eps)$ for all $t \in
    [0,T]$.
  \item $\cM_{c}(\mu,\eps)$ contains the set of bounded solutions of \cref{eq:slowFastInf}
    staying in $O_u\times O_v$ for all $t\in\mathbb R$, that is, if $(u,v)$ is a
    solution of \cref{eq:slowFastInf} satisfying $(u(t),v(t)) \in O_u \times O_v$ for
    all $t \in \RSet$, then $(u(0),v(0)) \in \cM_{c}(\mu,\eps)$.
  \end{romannum}
\end{theorem}
\begin{proof} 
  We use the compact formulation of the $\infty$-fast, $m$-slow system
  \cref{eq:slowFastInf} as a differential equation on $\tilde \cX = \cX \times \RSet^m$,
  \begin{equation}\label{eq:uTildeODE}
    \frac{d}{dt} \tilde u = \tilde L \tilde u + \tilde R (\tilde u,\mu,\eps)
  \end{equation}
  where $\tilde u = (u,v)$ and
  \[
    \tilde L =
    \begin{pmatrix}
      L & D_{v}R(0,0,0,0) \\
      0      & 0
    \end{pmatrix},
    \qquad
    \tilde R( \tilde u,\mu,\eps) = 
    \begin{pmatrix}
      R(u,v,\mu,\eps) - D_vR(0,0,0,0)v\\
      \eps G(u,v,\mu,\eps)
    \end{pmatrix}
    \,.
  \]

  We aim to show that, if \crefrange{hyp:vectField}{hyp:linearEquation} hold, then
  the same hypotheses hold for $\tilde L, \tilde R$, on Banach spaces $\tilde X$,
  $\tilde Y$, $\tilde Z$, upon defining suitable projectors $\tilde P_c$, $\tilde
  P_{su}$. The assertion is then a consequence of Theorem 3.3 of \cite{Haragus2010},
  appllied
  to \cref{eq:uTildeODE}. 
  We proceed to verify the hypotheses in $3$ steps:

  \textit{Step 1. Verification of \cref{hyp:vectField} for system
      \cref{eq:uTildeODE}}. It is immediate to prove that \cref{hyp:vectField} and
      $G \in C^k(V_u \times V_v \times V_p \times V_\eps, \RSet^m)$ imply that
      \cref{hyp:vectField} holds also for $\tilde L, \tilde R$ in system
      \cref{eq:uTildeODE}, and we omit this proof. 

    \textit{Step 2. Verification of \cref{hyp:spectralDecomposition} for system
      \cref{eq:uTildeODE}}. 
      The expression of $\tilde L$ makes it easy to compute its resolvent as function
      of the one of $L$. We conclude that $\sigma_c(\tilde L) = \sigma_c(L) \cup
      \{0\}$.
      One consequence is that the spectrum of $\tilde L$ admits the decomposition
      \[
	\sigma(\tilde L) = \sigma_u \cup \tilde \sigma_c \cup \sigma_s,
      \]
      where $\sigma_s$, $\sigma_u$ are the stable and unstable spectra of $L$,
      respectively, and where
      $\tilde \sigma_c$ contains $n+m$ purely imaginary eigenvalues, where $n$ is as
      in \cref{hyp:spectralDecomposition}(i). 
      This guarantees that $\tilde \sigma_c$ contains finitely many eigenvalues.
      Next, we define the projectors associated with $\tilde \sigma_c$ as follows
      \begin{equation}\label{eq:projectors}
	\tilde P_c = 
	\begin{pmatrix}
	  P_c & 0              \\
	    0 & \id_{\RSet^m}
	\end{pmatrix},
	\qquad
	\tilde P_{su} = 
	\begin{pmatrix}
	  P_{su} & 0 \\
		0 &  0_{\RSet^m}
	\end{pmatrix},
      \end{equation}
      where $P_{c}, P_{su}$ are the spectral projectors associated with the
      centre eigenspace of $L$. We have $\tilde \cX = \tilde \cE_c \oplus \tilde \cX_{su}$, where
      \[
	\tilde \cE_c= \range \tilde P_c, \quad
	\tilde \cX_{su} = \ker \tilde P_c, \quad
	\dim \tilde \cE_c = \dim \cE_c + m< \infty.
      \]
      Since $\tilde \cE_0$ is finite dimensional, we conclude that the $\tilde n$
      eigenvalues in $\tilde \sigma_c$ have finite multiplicities, hence
      \cref{hyp:spectralDecomposition}(i) holds for $\tilde L$.
      Furthermore, since $\tilde L$ and $L$ have the same stable and unstable spectra,
      their spectral gaps coincide, that is, \cref{hyp:spectralDecomposition}(ii) for $L$
      implies that the same hypothesis holds for $\tilde L$, with the same $\gamma$.

  \textit{Step 3. Verification of \cref{hyp:linearEquation} for system
    \cref{eq:uTildeODE}}. We Set
  \[
    \tilde \cY_{su} = P_{su} \cY \quad = \cY_{su} \times \{0_{\RSet^m}\}, \qquad
    \tilde \cZ_{su} = P_{su} \cZ = \cZ_{su} \times \{0_{\RSet^m}\},
  \]
  and we denote by $\tilde L_{su}$ the restriction of $\tilde L$ to $\tilde Z_{su}$.
  To prove the assertion we must show that \cref{hyp:linearEquation} implies
  that the same hypothesis holds for the following differential equation on $\tilde
  \cX_{su}$
  \begin{equation}\label{eq:linProb}
    \frac{d}{dt} \tilde u_{su} = \tilde L_{su} \tilde u_{su} + \tilde f(t),
  \end{equation}
  where $\tilde f = (f,0_{\RSet^m}) \in C_\eta(\RSet,\tilde \cY_{su})$, with $f$ and
  $\eta$ given in \cref{hyp:linearEquation}.
%
  In the remainder of the proof, we use the symbol $0$ in place of $0_{\RSet^m}$ to
  simplify the notation. We consider the linear mapping
  \[
    \tilde K_{su} \colon \cZ_{su} \times \RSet^m \to \cZ_{su} \times \RSet^m, 
    \qquad
    (U,V) \mapsto (K_{su} U,0).
  \]
  It suffices to prove the following claims:
  \begin{enumerate}
    \item $\tilde K_{su} \tilde f = (K_{su} f,0)$ is an element of
      $C_\eta(\RSet,\tilde \cZ_{su})$.
    \item $\tilde K_{su} \tilde f$ is the unique solution to \cref{eq:linProb}.
    \item Let $\Pi(\eta)$ be fixed as in \cref{hyp:linearEquation}, then the following
      bound holds
    \[
      \Vert \tilde K_h \Vert_{ \mathcal{L}( C_\eta(\RSet,\tilde \cY_{su}) ,
					       C_\eta(\RSet,\tilde \cZ_{su}) )}
      \leq \Pi(\eta).
    \]
  \end{enumerate}
  To prove claim $1$, we recall that $\tilde \cZ_{su} = \cZ_{su} \times \{0\}$,
  therefore if $U \in \cZ_{su}$ then $(U,0) \in \cZ_{su}$ and  
  $\Vert U \Vert_{\cZ_{su}} = \Vert (U,0) \Vert_{\tilde \cZ_{su}}$. By
  \cref{hyp:linearEquation}, $(K_{su} f)(t) \in \cZ_{su}$ for all $ t
  \in \RSet$, and the mapping $t \mapsto (K_{su} f)(t)$ is continuous. This
  implies $\tilde K_{su} \tilde f = (K_{su} f,0) \in C^1(\RSet,\tilde \cZ)
  \subset C^0(\RSet,\tilde \cZ)$ and
  \[
    \begin{aligned}
      \Vert \tilde K_{su} \tilde f \Vert_{C_\eta(\RSet,\tilde \cZ_{su})} 
      & = \sup_{t \in \RSet} \e^{-\eta |t|} 
		    \Vert (\tilde K_{su} \tilde f)(t) \Vert_{\tilde \cZ_{su}} \\
    & = \sup_{t \in \RSet} \e^{-\eta |t|} \Vert (K_{su}f)(t) \Vert_{\cZ_{su}} \\
    & = \Vert K_{su}f \Vert_{C_\eta(\RSet,\cZ_{su})},
    \end{aligned}
  \]
  which is finite by \cref{hyp:linearEquation}. Therefore $\tilde K_{su} \tilde f \in
  C_\eta(\RSet,\tilde \cZ_{su})$.

  In order to prove claim $2$ we note that
  \begin{align*}
    \frac{d}{dt} \tilde K_{su} \tilde f 
    & = 
    \frac{d}{dt} 
    \begin{pmatrix}
      K_{su} f\\
      0
    \end{pmatrix} \\
    & =
    \begin{pmatrix}
      L_{su} K_{su} f + f(t)\\
      0
    \end{pmatrix}
    && \text{by \cref{hyp:linearEquation},}
    \\
    & = \tilde L_{su} \tilde K_{su} \tilde f + \tilde f(t),
  \end{align*}
  therefore $\tilde K_{su} \tilde f$ solves \cref{eq:linProb}. Uniqueness is
  proved by showing that any solution $\tilde U = (U,V) \in C_\eta(R,\tilde
  \cZ_{su})$ to \cref{eq:linProb} is equal to $\tilde K_{su} \tilde f$.
  Since $\tilde \cZ_{su} = \cZ_{su} \times \{0\}$, we have
  $V(t) \equiv 0$ for all $t \in \RSet$, hence $\tilde U =(U,0)$.
  Since $\tilde U$ solves \cref{eq:linProb}, we have 
  \[
    \frac{d}{dt} U = L_{su} U + f(t), 
  \]
  therefore $U$ solves \cref{eq:linProbFastSubs}, and
  \cref{hyp:linearEquation} guarantees $U = K_{su} f$. We conclude
  $\tilde U = (K_{su} f,0) = \tilde K_{su} \tilde f$, hence $\tilde K_{su}
  \tilde f$ is the unique solution to \cref{eq:linProb}.

  To prove claim $3$, we estimate
  \[
  \begin{aligned}
    \Vert \tilde K_{su} & \Vert_{\mathcal{L}( C_\eta(\RSet,\tilde \cY_{su}) , 
                                               C_\eta(\RSet,\tilde \cZ_{su}) )} \\
        &= \sup 
	\big\{ \Vert \tilde K_{su} \tilde U \Vert_{C_\eta(\RSet,\tilde \cZ_{su})} 
	      \colon \Vert \tilde U \Vert_{C_\eta(\RSet,\tilde \cY_{su})} = 1 \big\}
	       \\
        &= \sup 
	\big\{ \Vert \tilde K_{su} (U,V) \Vert_{C_\eta(\RSet,\tilde \cZ_{su})} 
	\colon \Vert (U,V) \Vert_{C_\eta(\RSet,\tilde \cY_{su})} = 1 \big\} \\
        &= \sup 
	\big\{ \Vert K_{su} U \Vert_{C_\eta(\RSet,\tilde \cZ_{su})} 
	\colon \Vert (U,V) \Vert_{C_\eta(\RSet,\tilde \cY_{su})} = 1 \big\} \\
        &= \sup 
	\big\{ \Vert K_{su} U \Vert_{C_\eta(\RSet,\cZ_{su})} 
	\colon \Vert (U,V) \Vert_{C_\eta(\RSet,\tilde \cY_{su})} = 1 \big\} \\
        &\leq 
	\Vert K_{su} \Vert_{\mathcal{L}( C_\eta(\RSet,\cY_{su}) , C_\eta(\RSet,\cZ_{su}) )}
	\sup 
	\big\{ \Vert U \Vert_{C_\eta(\RSet,\cY_{su})} 
	\colon \Vert (U,V) \Vert_{C_\eta(\RSet,\tilde \cY_{su})} = 1 \big\} \\ 
	&\leq \Vert K_{su} \Vert_{\mathcal{L}( C_\eta(\RSet,\cY_{su}) ,
	  C_\eta(\RSet,\cZ_{su}) )} \\
	&\leq \Pi(\eta),
  \end{aligned}
  \]
  which completes this step.

  Following the $3$ steps above, the assertion is a consequence of the
  parameter-dependent version of the centre-manifold theorem \cite[Theorem
  3.3]{Haragus2010} for system \cref{eq:uTildeODE}. This theorem guarantees the
  existence of neighbourhoods $O_{\tilde u} \times O_\mu \times O_\eps$ of $0 \in
  \tilde \cZ \times \RSet^q \times \RSet$, a reduction function $\tilde \Psi \in
  C^k(\tilde E_c \times \RSet^q \times \RSet,\tilde Z_h)$ and a centre manifold 
  \[
  \cM_{c}(\mu,\eps) = \big\{ \tilde u_c + \tilde \Psi(\tilde u_c,\mu,\eps)
  \colon \tilde u_c \in \tilde E_c \big\},
  \]
  that satisfy the assertion. The expression above, however, is seemingly different
  from \cref{eq:CMDefinition}. To recover \cref{eq:CMDefinition} and the assertion in
  terms of a reduction function $\Psi \in C^k(E_c \times \RSet^m, \RSet^q \times
  \RSet, \cZ_{su})$, we note that $\tilde \Psi$ is with values in $\tilde \cZ_{su}=
  \cZ_{su} \times \{0\}$ hence $\tilde \Psi = (\Psi,0)$, for some $\Psi \in
  C^k(E_c \times \RSet^m \times \RSet^q \times \RSet, Z_{su})$, and this completes
  the proof.
\end{proof}

After the existence of a centre manifold has been derived, we proceed to write the
reduced equation on the centre manifold, as the following corollary states.
\begin{corollary}\label{cor:reducedEquation}
  Assume the hypotheses of \cref{thm:centreManifoldWithParameters}. Let $T >0$, let
  $(\mu,\eps) \in O_\mu \times O_\eps$, and let $(u,v)$
  be a solution to \cref{eq:slowFastInf} which belongs to $\cM_{c}(\mu,\eps)$ for $t \in
  [0,T]$. Then $u = u_c+\Psi(u_c,v,\mu,\eps)$, with $(u_c,v)$ satisfying
	\begin{equation}\label{eq:reducedEquation}
	\begin{aligned}
	  \dot{u}_c &=& P_cF(u_c+\Psi(u_c,v,\mu,\eps),v,\mu,\eps),\\
	  \dot v    &=& \eps G(u_c+\Psi(u_c,v,\mu,\eps),v,\mu,\eps).\\
	\end{aligned}
	\end{equation}
\end{corollary}
\begin{proof}
  Let $\tilde u = (u,v)$ be a solution to 
  \cref{eq:slowFastInfComp}, the compact formulation of the $m$ slow $\infty$-fast
  problem, and assume $\tilde u$ stays in $\cM_c$ for $t \in [0,T]$.
  Then
  \[
    \tilde u = 
    \begin{pmatrix}
      u \\
      v
    \end{pmatrix}
    =
    \begin{pmatrix}
      u_c + \Psi(u_c, v, \mu, \eps) \\
      v
    \end{pmatrix},
    \qquad
    u_c = P_c u.
  \]
  From \cref{eq:slowFastInfComp} and the definition of $\tilde P_c$ and $\tilde F$ we
  obtain
  \[
    \begin{aligned}
    \frac{d}{dt} 
    \begin{pmatrix}
      u_c \\
      v
    \end{pmatrix}
    & =
    \tilde P_c
    \frac{d}{dt} 
    \big( \tilde u_c + \tilde \Psi (\tilde u_c,\mu,\eps) \big) \\
    & = 
    \tilde P_c
    \tilde F\big( \tilde u_c + \tilde \Psi (\tilde u_c,\mu,\eps), \mu, \eps \big) \\
    & = 
    \tilde P_c
    \tilde F\big( ( u_c + \Psi (u_c,v,\mu,\eps), v), \mu, \eps \big) \\
    & = 
    \begin{pmatrix}
      P_cF(u_c+\Psi(u_c,v,\mu,\eps),v,\mu,\eps)\\
      \eps G(u_c+\Psi(u_c,v,\mu,\eps),v,\mu,\eps)
    \end{pmatrix}.
  \end{aligned}
  \]
\end{proof}

Importantly, the evolution equation \cref{eq:reducedEquation} is posed on a
finite-dimensional state space, whose dimension is $\dim E_c + m$. Owing to the
particular form of the right-hand side of \cref{eq:slowFastInfComp}, this
finite-dimensional dynamical system has $m$ slow and $\dim E_c$ fast variables. 

\begin{remark}
The centre manifold reduction theory outlined here applies directly to the finite
dimensional case, i.e.~all results have a finite dimensional version for a slow-fast
system \eqref{eq:slowFastFin}. We will highlight these results in later sections
where  we discuss folds and delayed Hopf (and Turing) bifurcations in detail.
\end{remark}

\subsection{Centre Manifold reduction with symmetries}\label{sec:reductionSymmetries}
When the fast subsystem admits symmetries, it is expected that they influence the
centre-manifold reduction. The following result combines a parameter-dependent
centre-manifold reduction for the $\infty$-fast, $m$-slow dynamical system (which has
been developed above) with equivariant centre-manifold theory for infinite
dimensional dynamical systems \cite{golubitsky1988singularities, chossat2000methods}.
We give this result without proof, as it is a minor modification of
\cref{thm:centreManifoldWithParameters}.

\begin{theorem}
  Assume there is a linear operator $T\in \mathcal{L}(\cX) \cap
  \mathcal L(\cZ)$ which commutes with the fast vector field
  \[
    T L=LT,\quad T R(u,v,\mu,\eps)=R(T u,v,\mu,\eps),
  \]
  and such that the restriction $T_c$ of $T$ to the
  subspace $E_c$ is an isometry. Under the assumptions in
  Theorem~\ref{thm:centreManifoldWithParameters}, there exists a reduction
  function $\Psi$ which commutes with $T$, that is, $T\Psi(u_c,v,\mu,\eps)
  =\Psi(T_c u_c,v,\mu,\eps)$ for all $u_c\in X_c$, and such that the vector field
  in the reduced equation \eqref{eq:reducedEquation} commutes with $(T_c,\id)$.
\end{theorem}


\section{Reduction around bifurcations of the fast subsystem} \label{sec:reductions}
We now specialise the result of the previous section, and find reductions of the full
systems around saddle-node, Hopf, and Turing bifurcations of the fast subsystem. We
will also derive normal form calculations for all these cases.

\subsection{Reduction around a fold of the critical manifold}\label{sec:reductionFold}
Before deriving a normal form around the fold of the fast subsystem, we make
additional assumptions that are verified at the saddle-node bifurcation point:
\begin{hypothesis}[Fold of the critical manifold]\label{hyp:fold}
Let $\left( e_i \right)_{i=1}^m$ be the canonical basis in $\RSet^m$. 
  \begin{romannum}
  \item The centre spectrum $\sigma_c$ of $L$ is given by $\sigma_c = \{0\}$. The eigenvalue
    $0$ has multiplicity $1$ and eigenvector $\zeta \in \cZ$.
  \item We have
    $P_c(D^2_u R(0,0,0,0) \zeta^{(2)}) \neq 0$.
    \item There exists $1\leq j \leq m$ such that  
      $P_c(D_vR(0,0,0,0) e_j)\neq0$.
  \end{romannum}
\end{hypothesis}
\begin{remark} When $L$ has an adjoint operator $L^*$ (this happens for instance when $\cZ$ is dense in
  $\cX$) then there is a general expression for the centre projector, namely $P_c(u) =
  \langle \zeta^* ,u\rangle_{\cX} \zeta$ where $\zeta^*$ is the eigenvector of $L^*$
  for the eigenvalue $0$ such that $\langle \zeta^* ,\zeta\rangle_{\cX} =1$.
\end{remark}

\begin{lemma}[Normal form for fold] \label{lem:FoldNormalForm}
  Assume \cref{hyp:fold}(i) and the hypotheses of \cref{cor:reducedEquation}, and let
  $O_\mu$, $O_\eps$ be defined as in \cref{cor:reducedEquation}. There exist
  neighbourhoods $O_A$, $O_B$ of $0$ in $\RSet$ and $\RSet^m$, respectively, such
  that, for any $(\mu,\eps) \in O_\mu \times O_\eps$ there is a change of variables,
  defined in $O_A \times O_B$ which transforms the reduced
  equation~\cref{eq:reducedEquation} into 
\begin{equation}\label{eq:CMFold}
  \begin{aligned}
    \dot A & = \alpha(B) + \beta A^2 + \gamma(\mu,\eps) 
+ O(A^2(\norm{B}+|\mu|+|\eps|)+(\norm{B}+|\mu|+|\eps|)^2)\\
    \dot B & = \eps G\big(A \zeta + \Psi(A,B,\mu,\eps),B,\mu,\eps\big) \\
  \end{aligned}
\end{equation}
where $A(t) \in \RSet$, $B(t) \in \RSet^m$ and
\begin{align*}
\alpha(B)\zeta & = P_c(D_vR(0,0,0,0) B), 
\\
\beta\zeta & =P_c(D^2_u R(0,0,0,0) \zeta^{(2)} ), 
	        	\\
\gamma(\mu,\eps)\zeta & = P_c(D_\mu R(0,0,0,0) \mu + D_\eps R(0,0,0,0) \eps ).
\end{align*}
If, in addition, \cref{hyp:fold}(ii) and \cref{hyp:fold}(iii) hold, we have $\alpha
\neq 0$ and $\beta \neq 0$, respectively.
\end{lemma}
\begin{proof}
  We fix $(\mu,\eps) \in O_\mu \times O_\eps$. From $X_c=\spn(\zeta)$, we write
  $( u_c, v ) = (A \zeta, B)$,  where $A \in \RSet$, $B \in \RSet^m$. This parametrization is
  defined in a neighbourhood $O_A \times O_B$ of $0 \in \RSet \times \RSet^m$. From
  \Cref{eq:reducedEquation} and the definition of $F$, we obtain
  \begin{equation}\label{eq:reducedEquationRepeat}
  \begin{aligned}
    \dot{A} \zeta &= P_c R(A \zeta+\Psi(A \zeta,B,\mu,\eps),B,\mu,\eps)\\
    \dot B    &= \eps G(A \zeta+\Psi(A \zeta,B,\mu,\eps),B,\mu,\eps).
  \end{aligned}
  \end{equation}
  The second equation in the system above is already written as in \Cref{eq:CMFold}
  of the statement, therefore we concentrate henceforth only on the first equation.
  
  In the remainder of this proof we set $\Vert \blank \Vert = \Vert \blank
  \Vert_{\cZ}$, and with a small abuse of notation we use $| \blank |$ for the 2-norm
  in both $\RSet^p$ and $\RSet^m$. We consider the Taylor expansion of $R$ around
  the origin 
  \[
    R(u,v,\mu,\eps) 
    = \sum_{0 \leq i+j+r+s \leq 2} R_{ijrs}(u^{(i)},v^{(j)},\mu^{(r)},\eps^{(s)}) + O\bigg( \sum_{i+j+r+s=3} \Vert u
    \Vert^i |v|^j |\mu|^r \eps^s \bigg),
  \]
  and we recall that $R_{0000} = 0$, $R_{1000}(u) = 0$ by \cref{hyp:vectField}, hence
  \[
    \begin{aligned}
    R(u,v,\mu,\eps) = R_{0100}(v) + R_{0010}(\mu) + R_{0001}(\eps) & + R_{2000}(u^{(2)})  \\
    & + O\bigg( \sum_{\substack{i+j+r+s=2,\\ i \neq 2}} \Vert u
    \Vert^i |v|^j |\mu|^r \eps^s \bigg)
    \end{aligned}
  \]
  Using the following identities
  \begin{align*}
    & R_{0100}(v) = D_vR(0,0,0,0) v, && R_{0010}(\mu) = D_\mu R(0,0,0,0) \mu, \\
    & R_{0001}(\eps) = D_\eps R(0,0,0,0) \eps, && R_{2000} (u^{(2)}) = \frac{1}{2} D^2_uR(0,0,0,0) u^{(2)}
  \end{align*}
and recalling that $\Psi = O(\Vert u \Vert^2+|v|^2+|\mu|+\eps)$, we obtain
\[
  \begin{aligned}
  \dot{A}\zeta = P_c(D_vR(0,0,0,0) B ) 
  & + A^2 P_c(D^2_uR(0,0,0,0) \zeta^{(2)} )  \\
  & +P_c(D_\mu R(0,0,0,0) \mu + D_\eps R(0,0,0,0) \eps) \\
  & + O\bigg( \sum_{\substack{i+j+r+s=2,\\ i \neq 2}} \Vert u \Vert^i |v|^j |\mu|^r \eps^s \bigg)
  \end{aligned}
\]
which, combined with the second equation in \cref{eq:reducedEquationRepeat} proves
that \cref{eq:CMFold} holds in $O_A \times O_B$. \Cref{hyp:fold}(iii) guarantees
$\beta \neq 0$ and \Cref{hyp:fold}(iv) implies that $\alpha$ is nonzero in $O_B
\setminus \{0\}$.
\end{proof}

\begin{remark} \label{rem:FoldNormalForm}
  Note that, with a further change of coordinates, \Cref{eq:CMFold} can
  be transformed into 
  \[
  \begin{aligned}
    \frac{d}{dt} A & = \alpha \tilde B_1 + \beta A^2 
    + O\bigg( \sum_{\substack{i+j+r+s=2,\; i \neq 2}} A^i |\tilde B|^j |\mu|^r \eps^s \bigg) \\
  \frac{d}{dt} \tilde{B} & = \eps \tilde G\big(A \zeta + \Psi(A,\tilde B,\mu,\eps), \tilde B,\mu,\eps\big)
  \end{aligned}
  \]
\end{remark}

The reduction above leads to studying an ODE with $1$ fast and $m$ slow variables,
near the fold. 
This system is amenable to applying existing results on
canards in this context. We mention here classical results on canards for $m=1$
\cite{BCDD1981,DR96,KS01a,KS01b}, for $m =2$ near a folded singularity 
\cite{SW01,W05,BKW06} and their generalisation to the case of $m$ slow variables
\cite{W12}.

\subsection{Reduction around a Hopf bifurcation in fast subsystem}\label{sec:reductionHopf}
\begin{hypothesis}[Hopf bifurcation]\label{hyp:Hopf}
  The centre spectrum $\sigma_c$ of $L$ consists of a simple pair of complex
    conjugate purely imaginary eigenvalues, $\sigma_c = \{i \omega, -i\omega\}$,
    $\omega {>0}$, with associated eigenvectors $\{\zeta,\bar \zeta\}$.
\end{hypothesis}

\begin{lemma}[Normal form of Hopf bifurcation]\label{lem:HopfNormalForm}
  Assume \cref{hyp:Hopf} and the hypotheses of \cref{cor:reducedEquation}. Then there
  exist neighborhoods $W_u, W_v, W_\mu, W_\eps$ in $O_u,O_v,O_\mu,O_\eps$,
  respectively, such that for any $(\eps,\mu) \in
  W_\mu \times W_\eps$ there exists a $(\mu,\eps)$-dependent polynomial change of
  variables 
  \begin{equation}\label{eq:HopfChangeOfVariables}
    \begin{aligned}
  u_c & = A\zeta + \bar A \bar \zeta+ \Phi(A\zeta + \bar A \bar \zeta,B,\mu,\eps) \\
  v   & = B
   \end{aligned}
  \end{equation}
  with $A\in \CSet$, $B\in \RSet^m$, $\Phi(\blank,v,\mu,\eps)$ a polynomial of degree
  $k$, $\Phi(0,0,0,0) = 0$, $D_u\Phi(0,0,0,0) = 0$, and whose monomials of
  degree $q$ are functions of $(v,\mu,\eps)$ of class $C^{k-q}$, which transforms the reduced
  equation~\cref{eq:reducedEquation} for
  $(u_c,v) \in W_u \times W_v$ into 
  \begin{equation}\label{eq:HopfNormalForm}
    \begin{aligned}
      \dot A & = i \omega A + \alpha(B,\mu,\eps)A + \beta A |A|^2 + O\big( (|B| + |\mu| + \eps + |A|^2 )^2 \big) \\
      \dot B &= \eps G(A \zeta+\bar A \bar \zeta+\Phi(A \zeta+\bar A \bar \zeta,B,\mu,\eps),B,\mu,\eps)
    \end{aligned}
  \end{equation}
  where $\alpha \colon \RSet^m \times \RSet^p \times \RSet \to \CSet$ is multilinear
  and $\beta \in \RSet$.
\end{lemma}
\begin{proof}
  In this case, \Cref{eq:reducedEquation} leads to
  \begin{equation}\label{eq:reducedEquationRepeated}
  \begin{aligned}
    \dot{u}_c &= L_c u_c + P_c R(u_c+\Psi(u_c,v,\mu,\eps),v,\mu,\eps),\\
    \dot v    &= \eps G(u_c+\Psi(u_c,v,\mu,\eps),v,\mu,\eps).\\
  \end{aligned}
  \end{equation}
  with $L_c u_c \neq 0$. The operators $L_c$ and $L$ have an identical centre
  spectrum, and by \cref{hyp:Hopf} we have $E_c= \spn\{\zeta,\bar \zeta\}$, hence
  $\dim E_c = 2$. As in the case of saddle-node bifurcations, we write a normal form
  for the first of the two equations above, when $v$ is fixed. We set $v=B$ and
  consider 
  \[
    \dot{u}_c = L_c u_c + P_c R(u_c+\Psi(u_c,B,\mu,\eps),B,\mu,\eps).
  \]
  We use a parameter-dependent normal form result \cite[Theorem 2.2, Chapter
  3]{Haragus2010} on the $2$-dimensional system above, with parameters $(B,\mu,\eps)
  \in \RSet^m \times \RSet^p \times \RSet$. There exist neighbourhoods 
  $W_u, W_v, W_\mu, W_\eps$ in $O_u,O_v,O_\mu,O_\eps$
  respectively, and a polynomial $\Phi$ such that, for any $(B,\mu,\eps) \in W_v,
  \times W_\mu \times W_\eps$, the change of variables $u_c = \nu_c +
  \Phi(\nu_c,B,\mu,\eps)$ transforms the system above into the normal form
  \[
    \dot \nu_c = \Lambda \nu_c + N(\nu_c,B,\mu,\eps) + o(|(B,\mu,\eps)|^q)
  \]
  for any integer $q$ with $2 \leq q \leq k$, where $k$ is fixed as in
  \cref{hyp:vectField}(ii). The fact that $\Phi$ is a polynomial, and the properties
  listed in the statement for $\Phi$ are also a consequence of \cite[Theorem 2.2,
  Chapter 3]{Haragus2010}.
  We aim to recover the normal form for a Hopf bifurcation,
  hence it is convenient to identify $\RSet^2$ with the diagonal $\{(z,\bar z) \colon
  z \in \CSet\}$, thereby representing $\nu_c$ as $\nu_c(t) = A(t) \zeta + \bar A(t) \bar \zeta$
  and obtaining \cite[Lemma 1.7, Chapter 3]{Haragus2010}
  \[
    \frac{d}{dt}
    \begin{pmatrix}
      A \\
      \bar A \\
    \end{pmatrix}
    = 
    \begin{pmatrix}
      i \omega &           0 \\
      0        & -i \omega  
    \end{pmatrix}
    \begin{pmatrix}
      A \\
      \bar A \\
    \end{pmatrix}
    +
    \begin{pmatrix}
      \alpha(B,\mu,\eps) A + \beta A |A|^2 \\
      \alpha(B,\mu,\eps) \bar A + \beta \bar A |A|^2
    \end{pmatrix}
    + O\big( (|B| + |\mu| + \eps + |A|^2 )^2 \big) \\
  \]
  where $\alpha$ is a linear form and $\beta$ a real number. We have therefore shown
  the existence of a change of coordinates of type \cref{eq:HopfChangeOfVariables}
  which transforms the reduced equation
  \cref{eq:reducedEquationRepeated} into \cref{eq:HopfNormalForm}. 
\end{proof}

\begin{remark} 
  We highlight that, in contrast to the reduction around a saddle-node bifurcation,
  we do not include here explicit expressions for the coefficients of the normal
  form $\alpha$, $\beta$, and we do not include in \cref{hyp:Hopf} non-degeneracy
  conditions to guarantee $\alpha, \beta \neq 0$. These conditions are known for ODEs
  \cite{kuznetsov2013elements}, and computing the coefficients is strongly
  problem-dependent, therefore we omit such calculations in the present paper.
\end{remark}

The centre manifold reduction leads to a $C^k$ (and not analytic) ODE with $2$ fast and 
$m$ slow variables. Hence, this system is amenable to applying existing $C^k$ results on 
delayed loss of stability through a Hopf bifurcation. We refer to classical results for $m=1$, e.g.,
\cite{neishtadt87,neishtadt88,HKSW16}. These $C^k$ results guarantee the existence of a
short delay ($O(\eps\ln(\eps))$) due to the slow passage. 
By \cref{eq:HopfChangeOfVariables}, we deduce that to first order
the fast variables on the centre manifold are approximated by $A(t) \xi(x) +
\bar A(t) \bar \xi(x)$, where $\xi(x)$ is the unstable Hopf mode, and where $A(t)$ has the
dynamics of a slow passage through a Hopf bifurcation. The existence of long $O(1)$
delays associated with slow passage through a Hopf bifurcation in a PDE in the
analytic context is still an open problem.

\subsection{Reduction around Turing bifurcation}
In order to perform a reduction around a Turing bifurcation for the applications, we
make some additional assumptions concerning the symmetry of the problem.

\begin{hypothesis}[O(2)-equivariance] \label{hyp:O2Equivariance}
  There exists a one-parameter continuous family
 of linear maps $T_\phi$ on $X$, for $\phi\in
\RSet/2 l\ZSet$, $l \in \RSet$, and a symmetry $S$ on $X$ such that:
\begin{romannum}
  \item $T_\phi\circ T_\psi = T_{\phi+\psi}$ and $S T_\phi=T_{-\phi}S$ 
	for $\phi,\psi \in \RSet/2l \ZSet$ 
  \item $T_0=\id$ and $S^2=\id$
  \item The fast vector field $u \mapsto F(u,v,\mu,\varepsilon)$ commutes with
    $T_\phi,S$.
\end{romannum}
\end{hypothesis}

\begin{lemma}[Normal form of Pitchfork O(2) bifurcation] \label{lem:TuringNormalForm}
  Assume \cref{hyp:O2Equivariance}, and the hypotheses of \cref{cor:reducedEquation}.
  Further, assume that $0$ is a double eigenvalue of $L$ and no other eigenvalue is in
  $\sigma_c$ for $(\varepsilon,\mu)=(0,0)$. Assume that the action of $T_\phi$ on $X_c$ is not
  trivial. Then, there exist neighborhoods $O_u,O_v,O_\mu,O_\eps$ of $0$ in
  $\RSet^2$, $\RSet^m$, $\RSet^p$, and $\RSet$, respectively, such that for any
  $(\eps,\mu) \in O_\mu \times O_\eps$,  in a suitable basis $(\zeta,\bar\zeta)$ of $\ker L$
  \begin{equation}\label{eq:coordChangePitch}
  \begin{aligned}
  u_c & = A\zeta + \bar A \bar \zeta+ \Phi(A\zeta + \bar A \bar \zeta,B,\mu,\eps) \\
  v   & = B
  \end{aligned}
  \end{equation}
  with $A\in \CSet$, $B \in \RSet^m$, $\Phi(\blank,v,\mu,\eps)$ a polynomial of degree
  $k$, $\Phi(0,0,0,0) = 0$, $D_u\Phi(0,0,0,0) = 0$, and whose monomials of
  degree $q$ are functions of $(v,\mu,\eps)$ of class $C^{k-q}$. The reduced
  equation~\cref{eq:reducedEquation} for
  $(u_c,v) \in O_u \times O_v$ has, at third order in $(A,\bar A)$, the expression
  \begin{equation}\label{eq:normFormPitch}
  \begin{aligned}
  \dot A & = \alpha(B,\mu,\eps)A + \beta A |A|^2 + O\big( (|B| + |\mu| + \eps + |A|^2 )^2 \big) \\
  \dot B &= \eps G(A \zeta+\bar A \bar \zeta+\Phi(A \zeta+\bar A \bar \zeta,B,\mu,\eps),B,\mu,\eps)
  \end{aligned}
  \end{equation}
  where $\alpha,\beta$ are real multilinear forms.
\end{lemma}
\begin{proof}
This proof follows closely the one of Lemma~\ref{lem:HopfNormalForm} except that at
the end, we have to use the O(2)-pitchfork normal formal instead of the Hopf one. As
shown in \cite[Section 1.2.4]{Haragus2010}, one can find a suitable basis
$\zeta,\bar\zeta$ such that $T_\phi\zeta = e^{ik\phi}\zeta$ for some fixed
$k\in\mathbb N$ and $S\zeta=\bar\zeta$. In this
basis, we represent $\nu_c$ as
$\nu_c(t) = A(t) \zeta + \bar A(t) \bar \zeta$ and obtain at third order
\[
\dot A  = \alpha(B,\mu,\eps)A + N_3(A,\bar A,\mu,\eps)+ O\big( (|v| + |\mu| + \eps + |A|^2 )^2 \big).
\]
where $N_3$ is a cubic polynomial in $(A,\bar A)$ which commutes with the action of
the symmetry group. Thus $N_3(e^{i\phi}A,e^{-i\phi}\bar A,\mu,\varepsilon) =
e^{i\phi}N_3(A,\bar A,\mu,\varepsilon)$ and $\overline N_3(A,\bar A,\mu,\varepsilon) =
N_3(\bar A, A,\mu,\varepsilon)$.
This gives:
\[
\dot A  = \alpha(B,\mu,\eps)A + \beta A |A|^2 + O\big( (|v| + |\mu| + \eps + |A|^2 )^2 \big)
\]
where $\alpha,\beta$ are real valued.
\end{proof}

This lemma applies to the case of a PDE equivariant with respect to translations (which take the role of $T_\phi$) in one unbounded spatial direction and possesses a
reflection symmetry in this direction.
The reduction above leads to studying an ODE with $1$ fast and $m$ slow variables,
that is \Cref{eq:normFormPitch}. Results on slow-passages through transcritical and
pitchfork bifurcations for ODEs of this type are available in the literature
\cite{KS01c,boudjellaba2009dynamic}. For example, if $m=1$ and the coefficients $\alpha$ and $\beta$
are such that the layer problem undergoes a pitchfork bifurcation at the origin, then
the results in \cite{KS01c} imply that \Cref{eq:normFormPitch} has a delayed
pitchfork bifurcation. By \cref{eq:coordChangePitch}, we deduce that to first order
the fast variables on the centre manifold are approximated by $A(t) \xi(x) +
\bar A(t) \bar \xi(x)$, where $\xi(x)$ are Turing modes, and where $A(t)$ has the
dynamics of a slow passage through a pitchfork. This phenomenon is what we
call a slow passage through a Turing bifurcation.

\section{Applications} \label{sec:applications}
In this section we collect applications of centre-manifold reductions and normal
forms to several examples presented in \cref{sec:numericalExamples}. We highlight
that the order in which the material is presented does not follow the one used in
\cref{sec:numericalExamples}, because we have chosen to start from simple examples
and work towards more advanced material. 
In \cref{sec:preparatoryResults} we present preparatory results for Hilbert spaces,
which will be used in the following calculations. We then work through an example of
slow-passage through a Hopf bifurcation in \cref{eq:FN} in
\cref{sec:slowPassageHopfAnalysis}, assuming no prior knowledge
of centre-manifold reductions for infinite-dimensional systems. In
\cref{sec:slowPassageTuringAnalysis}, we discuss slow-passage through a Turing
bifurcation in the nonlocal reaction--diffusion equation~\cref{eq:nonlocalRD}, and
give pointers for performing the same analysis in local reaction--diffusion problems.
The technical treatment of these sections depends on the fact that Hilbert spaces have
been used in the accompanying examples. This motivates our presentation of the
example in \cref{sec:slowPassageFoldAnalysis}, which is posed on Banach spaces and
discusses canards in neural field models. Towards the end of this section, in
\cref{sec:AnalysisDDE}, we discuss slow passage through Hopf in the
DDE~\cref{eq:DDE}.

\subsection{Preparatory results} \label{sec:preparatoryResults}
We begin by introducing two lemmas that will be
useful in the upcoming calculations.

\begin{lemma}\label{lemma:resolvent-laplacian} Let $\Omega = (-l,l)$, for some $l>0$,
  and let $\cX = L^2(\Omega,\CSet)$ and $\cZ =
  H^2_\textrm{per}(\Omega,\CSet)$, 
  endowed with the standard norms.
  Further consider the Laplacian operator $\partial_x^2\in\mathcal L(\cZ,\cX)$. If
  $z\in\rho(\partial^2_x)$, then
\[
  \norm{(z - \partial^2_x)^{-1}}_{\mathcal L(\cX)}\leq \frac{1}{|\imag z|}.
\]
\end{lemma}
\begin{proof}
  We denote by $\Vert \blank \Vert_\cX$ and  $\langle \blank,\blank \rangle_\cX$ the
  standard norm and associated inner product on $\cX$, respectively. Let $
  z\in\rho(\partial^2_x)$. For any $f\in  X$, there is a
  unique solution $u \in \cZ$ to $\partial^2_xu-zu=f$. Multiplying the equation by
  the complex conjugate $\bar u$ of $u$,
  integrating over $\Omega$, and using integration by parts and boundary conditions
  we obtain
\[
  \norm{\partial_x u}_\cX^2-z\norm{u}_\cX^2 =  \langle
f, u\rangle_{\cX}.  \]
Taking the imaginary part of the equation above we obtain the bound
\[
|\imag z|\norm{u}_\cX^2 = |\imag \langle
f, u\rangle_{\cX}|\leq \norm{f}_\cX\norm{u}_\cX,
\qquad \textrm{hence} \qquad
\norm{u}_\cX\leq \frac{\norm{f}_\cX}{|\imag z|}
\]
and
\[
  \norm{(z - \partial^2_x)^{-1}}_{\mathcal L(\cX)} 
     = \sup_{\Vert f \Vert_\cX \neq 0}
     \frac{\norm{(z - \partial^2_x)^{-1} f}_\cX }
            { \norm{f}_\cX }
	    \leq \frac{1}{|\imag z|}.
\]
\end{proof}

\begin{remark}
  We note that \cref{lemma:resolvent-laplacian} can be derived in a similar fashion
  also when the problem is subject to to Neumann or Dirichlet boundary
  conditions. We omit these cases for simplicity.
\end{remark}

\begin{lemma}\label{lemma:bounded-pert}
  Let $A\in \mathcal L(\cZ,\cX)$, $B\in\mathcal L(\cX)$, and assume there exist
  $\omega_0, K> 0$ such that $\norm{(\iunit \omega
  - A)^{-1}}_{\mathcal L(\cX)}\leq K/|\omega|$ for all $\omega \in \RSet$ such that
  $|\omega|>\omega_0$. Then there exists $C>0$ such that
\[
\norm{(\iunit \omega - A-B)^{-1}}_{\mathcal L(\cX)}\leq \frac{C}{|\omega|}
\qquad \textrm{for all $|\omega|>\max(\omega_0,\norm{B}_{\mathcal{L}(X)}K)$.}
\]
\end{lemma}
\begin{proof}
Since $i\omega\in \rho(A)$ and $\norm{B}_{\mathcal{L}(X)}K<|\omega|$, then  $ \id  -B
(i\omega- A)^{-1}$ is invertible. We use the following identity
\[
(i\omega -A-B)^{-1} = (i\omega   - A)^{-1} \left[\id  -B (i\omega- A)^{-1}\right]^{-1}
\]	
and the main hypotheses to derive the bound
\[
  \norm{(\iunit \omega - A-B)^{-1}}_{\mathcal L(\cX)}\leq
  \frac{K}{|\omega|}\frac{1}{1-\norm{B}_{\mathcal{L}(X)}K/|\omega|}\leq  \frac{C}{|\omega|}.
\]
\end{proof}	

\subsection{Analytical results for slow-passage through a Hopf bifurcation in the
FitzHugh-Nagumo PDE} \label{sec:slowPassageHopfAnalysis}
We now show how to apply our framework to the FitzHugh-Nagumo model
\begin{equation}\label{eq:FNRev}
  \begin{aligned}
    \partial_t u_1 & = \partial_x^2 u_1 + u_1 -u_1^3/3 -u_2 + v, 
		   && (x,t) \in (0,2\pi), \times \RSet_{>0}, \\
    \partial_t u_2 & =  u_1 + c - b u_2,
		   && (x,t) \in (0,2\pi) \times \RSet_{>0}, \\
            \dot v & = \eps,                                     
		   && t \in \RSet_{>0}
  \end{aligned}
\end{equation}
subject to periodic boundary conditions
\[
  u_1(0,t) = u_1(2\pi,t), \qquad u_2(0,t) = u_2(2\pi,t), \qquad t \in \RSet_{> 0}.
\]
We make the following assumption, in order to guarantee the existence of a
homogeneous steady state undergoing a Hopf bifurcation.
\begin{hypothesis}\label{hyp:FNHopf}
  There exist $u_{*1},u_{*2}, v_*, b_*, c_* \in \RSet$ such that
  $u(x,t)\equiv (u_{*1},u_{*2})$ is an equilibrium of \cref{eq:FNRev} for
  $(v,b,c,\eps) = (v_*,b_*,c_*,0)$, with $1-(u_{*1})^2 = b_* \in (0,1)$.
\end{hypothesis}
In order to cast the problem in the framework of the previous sections, we set
$u = (u_1,u_2)$, $\mu = (b,c)$, apply the change of variables $u = u_* + \tilde u$,
$v = v_* + \tilde v$, $\mu = \mu_* + \tilde \mu$, and obtain, after dropping tildes
and setting $a = 1 -u_{*1}^2 = b_*$ for notational convenience,
\begin{equation}\label{eq:slowFastInfRep}
  \begin{aligned}
  \dot u & = L u + R(u,v,\mu,\eps) \\
  \dot v & = G(u,v,\mu,\eps)
  \end{aligned}
\end{equation}
where
\begin{equation}\label{eq:LRGDef}
  L = 
  \begin{pmatrix}
    \partial_x^2 + a \id  & -\id \\
     \id                  & -a \id
   \end{pmatrix},
   \qquad
   R(u,v,\mu,\eps) = o(\norm{u}),
   \qquad
   G(u,v,\mu,\eps) = \eps,
\end{equation}
and $R$ is the remainder, linear in $v$ and $\mu$, and independent of $\eps$.


System~\cref{eq:slowFastInfRep} has a homogeneous equilibrium $u=0$ at $(v,\mu,\eps)
= (0,0,0)$ and, as we shall see below, \cref{hyp:FNHopf} ensures that $u=0$,
considered as an equilibrium of the fast subsystem, undergoes a Hopf bifurcation for
these parameter values. We now proceed to check the hypotheses of
\cref{thm:centreManifoldWithParameters}.

\subsubsection{Choice of function spaces}
We begin by selecting function spaces for the problem. In particular, we set
\[
  \cZ = H^2_\textrm{per}(0,2\pi) \times L^2_\textrm{per}(0,2\pi),
  \qquad
  \cY = \cX = L^2(0,2\pi) \times L^2(0,2\pi).
\]
where $L^2(0,2\pi)$ and $H_\textrm{per}^2(0,2\pi)$ are standard Sobolev spaces
endowed with the inner-product norms $\Vert \blank \Vert_{L^2}$ and $\Vert \blank
\Vert _{H^2}$, respectively. We endow $X$ and $Z$ with norms
\[
  \Vert (u_1,u_2) \Vert_{Z} = \big( \Vert u_1 \Vert^2_{H^2} + \Vert u_2 \Vert^2_{L^2}
  \big)^{1/2},
  \qquad
  \Vert (u_1,u_2) \Vert_{X} = \big( \Vert u_1 \Vert^2_{L^2} + \Vert u_2 \Vert^2_{L^2}
  \big)^{1/2}.
\]
However, since $u_1, u_2$ are $2\pi$-periodic functions defined on $(0,2\pi)$, it
will be convenient to write them in terms of their Fourier Series
\[
  u_i(x) = \sum_{n\in \ZSet} u_{i,n} \psi_n(x), 
  \quad i = 1,2, \qquad \psi_n(x) = \frac{1}{\sqrt{2
  \pi}} e^{inx}, \quad n \in \ZSet.
\]
and use for $\cX$, $\cZ$ norms which expose the respective Fourier
coefficients~\cite{Atkinson:2005hy}
\[
  \Vert (u_1,u_2) \Vert_{\cX} = \sum_{n \in \ZSet} \Vert (u_{1,n},u_{2,n}) \Vert_2^2
  , 
  \qquad
  \Vert (u_1,u_2) \Vert_{\cZ} \sim \sum_{n \in \ZSet} n^4 |u_{1,n}|^2 + |u_{2,n}| ^2.
\]
In passing we note that the latter norm is not equal, but only equivalent, to $\Vert
\blank \Vert_{\cZ}$.

\subsubsection{Checking \cref{hyp:vectField}} To check \cref{hyp:vectField}(i) we
prove that $L$ is a continuous linear
operator from $Z$ to $X$. We fix $u \in \cZ$, let $f=Lu$, and obtain the following
bound using Cauchy-Schwarz inequality:
\[
  \begin{aligned}
  \Vert f_1 \Vert^2_{L^2} 
          & \leq 
	  \big( 
	    \Vert \partial_x^2 u_1 \Vert_{L^2} 
	     + a \Vert u_1 \Vert_{L^2} + \Vert u_2 \Vert_{L^2} 
	   \big)^2 \\
	   & \leq (2+a^2) 
	   \big(
	    \Vert \partial_x^2 u_1 \Vert^2_{L^2} 
	    + \Vert u_1 \Vert^2_{L^2} + \Vert u_2 \Vert^2_{L^2} \big) 
	  := K^2_1 \Vert u \Vert_\cZ^2.
  \end{aligned}
\]
Similarly we find $\Vert f_2 \Vert^2_{L^2} \leq K_2^2 \Vert u
\Vert^2_\cZ$, where $K_2^2 = 1+a^2$. Combining the bounds we obtain
$\Vert Lu \Vert_\cX \leq K \Vert u \Vert_\cZ$, where $K = \sqrt{ \max(K_1,K_2) }$.
This proves $L \in \mathcal{L}(Z,X)$.

Further, $R$ is a cubic polynomial in $X$ with $R(0,0,0,0) = 0$ and $D_uR(0,0,0,0) =
0$. Therefore, owing to the fact that $Z$ is a Banach algebra,
\cref{hyp:vectField}(ii) holds for any integer $k$.

\subsubsection{Checking \cref{hyp:spectralDecomposition}} To check this hypothesis,
we determine the resolvent set of $L$, as in the following statement.
\begin{proposition} \label{prop:FNResolvent}
  Let $D_n(\lambda) = \lambda^2 + n^2 \lambda + a(n^2-a) +1$, for $\lambda \in
  \CSet$ and $n \in \ZSet$. The resolvent set of the operator $
  L \colon Z \to X$ defined in \cref{eq:LRGDef} is given by
  \[
    \rho(L) = \{ \lambda \in \CSet \colon \lambda \neq -a, D_n(\lambda) \neq 0 \textrm{
    for all $n \in \ZSet$} \}.
  \]
\end{proposition}
\begin{proof} 
Let $E= \{ \lambda \in \CSet \colon \lambda \neq -a, D_n(\lambda) \neq 0
  \textrm{ for all $n \in \ZSet$} \}$. We show that $E \subseteq \rho(L)$ and $E
  \supseteq \rho(L)$, hence $\rho(L) = E$.

  \emph{Step 1: $E \subseteq \rho(L)$}. We prove that $\lambda \neq -a$ and
  $D_n(\lambda) \neq 0$ for all $n \in \ZSet$ imply $\lambda \in \rho(L)$, that is,
  for all $f \in \cX$,
  there exist a constant $K(\lambda)>0$ and a unique $u \in \cZ$ satisfying $(\lambda-L)u
  =f$ and $\Vert u \Vert_\cZ \leq K(\lambda) \Vert f \Vert_\cX$, hence $(\lambda 
  -L)^{-1} \in \mathcal{L}(X,Z)$. We fix $f \in \cX$, and we note that $(\lambda
  - L)u = f$ if, and only if,
  \begin{equation}\label{eq:FourierLinSys}
    \begin{pmatrix}
      \lambda + n^2 -a &  1           \\
                    -1 &  \lambda + a
    \end{pmatrix}
    \begin{pmatrix}
      u_{1,n} \\
      u_{2,n}
    \end{pmatrix}
    =
    \begin{pmatrix}
      f_{1,n} \\
      f_{2,n} 
    \end{pmatrix}
    \qquad
    \textrm{for all $n \in \ZSet$}.
  \end{equation}
  Since $D_n(\lambda) \neq 0$ is the determinant of the matrix above, the system has a
  unique solution for all $n \in \ZSet$. We now proceed to show that $u = \sum_{n \in
  \ZSet}
  u_n \psi_n(x) \in Z$. From the system above we get
  \begin{equation}\label{eq:interm}
  \begin{aligned}
    D_n(\lambda) u_{1,n} & = (\lambda + a) f_{1,n} - f_{2,n} \,,\\
    (\lambda + a) u_{2,n} & = ( u_{1,n} + f_{2,n})\,.
  \end{aligned}
  \end{equation}
  The first equation leads to the bound
  $
  |D_n(\lambda) u_{1,n} |^2 \leq (1 + |\lambda + a|^2) \Vert f_n \Vert_2^2.
  $
  We note that $D_n(\lambda) \sim n^2$ as $n \to \infty$: there exist an integer
  $n_0(\lambda)$ and a positive real constant $C(\lambda)$ such that $C(\lambda)
  |D_n(\lambda)| \geq n^2$ for all $n > n_0(\lambda)$, from which we
  deduce \[
    n^4 |u_{1,n}|^2 \leq C(\lambda)^{-2}
        (1 + |\lambda + a|^2) \Vert f_n \Vert_2^2 =
    K_1(\lambda) \Vert f_n \Vert^2_2, \qquad \textrm{for all $n > n_0(\lambda)$}.
  \]
  Since $\lambda \neq -a$, the second equation in \cref{eq:interm} gives the following
  bound
  \[
    \begin{aligned}
      |u_{2,n}|^2 
              & \leq 2|\lambda + a|^{-2} \big( |u_{1,n}|^2 + |f_{2,n}|^2 \big) & \\
	      & \leq 2|\lambda + a|^{-2} \big( K_1(\lambda) + 1 \big)
	                                                   \Vert f_n \Vert^2_2 & \\
	      & = K_2(\lambda) \Vert f_n \Vert^2_2 & \textrm{for all $n > n_0(\lambda)$}.
    \end{aligned}
  \]
  therefore
  $n^4 |u_{1,n}|^2 + |u_{1,n}|^2 
      \leq \max\big(K_1(\lambda),K_2(\lambda)\big) \Vert f_n \Vert^2_2$ 
  for all $n > N_1(\lambda)$ and owing to $f \in \cX$, we conclude that there exists
  $K(\lambda)$ such that $\Vert u \Vert_{\cZ} \leq K(\lambda) \Vert f \Vert_{\cX}$.
  For all $n$, $u_n$ is the unique solution to \cref{eq:FourierLinSys}, hence the
  uniqueness of the Fourier series implies uniqueness of $u$.

  \emph{Step 2: $E \supseteq \rho(L)$}. Equivalently, we prove $\CSet \setminus
  E \subseteq  \CSet \setminus \rho(L) = \sigma(L)$. If there exist $\lambda \in
  \CSet$ and $n \in \ZSet$ such that $D_n(\lambda) = 0$, then \cref{eq:FourierLinSys}
  does not have a solution, hence $\lambda \in \sigma(L)$. If $\lambda=-a$, then
  $(\lambda - L)u = f$ does not have a solution in $\cZ$ for all $f \in \cX$,
  because $\lambda = -a$ implies $u_1 = -f_2 \in L^2(0,2\pi)$, hence $u_1 \not\in
  H^2_\textrm{per}(0,2\pi)$ in general. We therefore have $-a \in \sigma(L)$.
\end{proof}

\Cref{prop:FNResolvent} and \cref{hyp:FNHopf} imply
\[
  \sigma(L) = 
  \Big\{\pm \iunit \sqrt{1-a^2},-a \Big\} 
  \cup
  \bigg\{ 
    \frac{-n^2 + \sqrt{(n^2-2a)^2 -4}}{2} \colon n \in \ZSet_{\neq 0}
  \bigg\}
\]
therefore the homogeneous equilibrium $u \equiv 0 \in X$ of the fast sybsystem of
\cref{eq:slowFastInfRep} undergoes a Hopf bifurcation at $(v,b,c) = (0,0,0)$, with
Hopf frequency $\omega_* = \sqrt{1-a^2} = \sqrt{1-(b^*)^2}$. It can be shown (we omit
this calculation for brevity)
\[
  \real \sigma_s(L) \subset (-\infty,-a/2), \qquad \sigma_c(L) = \{\pm \iunit \omega_*\},
  \qquad \sigma_u(L) = \{ \varnothing \},
\]
therefore \cref{hyp:spectralDecomposition} holds with $n=2$ and $\gamma = a/2$.
%

\subsubsection{Checking \cref{hyp:linearEquation}}
\label{subsec:checkResolventHypHopf}
Checking this hypothesis directly may
be challenging in applications. However, in a wide class of problems, this hypothesis
can be replaced by some other conditions, which are easier to check. In particular,
since $\cX = \cY$ and $\cX, \cY, \cZ$ are Hilbert spaces, it suffices to prove the following
condition~\cite[Section 2.2.3, Theorem 2.20]{Haragus2010}.

\begin{proposition}\label{prop:boundResolventHopf}
  Let $L \in \mathcal{L}(\cZ,\cX)$ be the linear operator defined in
  \cref{eq:LRGDef}. There exist $\omega_0, K_0 >0$ such that, for all $\omega \in \RSet$
  with $|\omega| > \omega_0$, $\iunit \omega \in \rho(L)$ and
  \begin{equation}\label{eq:FNLinHyp}
    \Vert (\iunit \omega - L)^{-1} \Vert_{\mathcal{L}(\cX,\cX)}
    \leq
    K_0/|\omega|
  \end{equation}
\end{proposition}
\begin{proof}
  We note that $\iunit \omega \in \sigma(L)$ if, and only if, 
    $
    | \omega | = \sqrt{ 4 - (n^2-2a)^2}/2 < 1,
    $
  hence $|\omega| > 1$ implies $\iunit \omega \in \rho(L)$. Henceforth we set
  $\omega_1=1$ for notational convenience. We now fix $\omega$
  with $| \omega | > \omega_1$, $f \in \cX$, and prove that there exist
  $K_0,\omega_0$ such that
  \begin{equation}\label{eq:intermediateBound2}
    \Vert (\iunit \omega - L)^{-1} f \Vert_{\cX} \leq 
    K_0 \Vert f \Vert_{\cX} / |\omega|
  \end{equation}
  which implies $\cref{eq:FNLinHyp}$. Since $\iunit \omega \in \rho(L)$, there exists
  a unique $u \in \cZ$ such that $u= (i\omega -L)^{-1} f$, hence
  \begin{equation}\label{eq:FNHopfRes}
      -\partial_x^2 u_1 +
    \bigg(
      \iunit \omega - a + \frac{1}{\iunit \omega + b} 
    \bigg) u_1 =
    f_1 - \frac{f_2}{\iunit \omega + b},
    \qquad
    u_2 = \frac{u_1 + f_2}{\iunit \omega + b}.
  \end{equation}
  We use \cref{lemma:resolvent-laplacian} to bound $u_1$: the operator $\partial_x^2$
  in \cref{eq:FNHopfRes} maps from $\colon H^2_\textrm{per}( (0,2\pi),\CSet)$ to
$L^2( (0,2\pi),\CSet)$, and
  \[
    \bigg( \iunit \omega - a + \frac{1}{\iunit \omega + b} \bigg) \in
    \rho(\partial_x^2),
    \qquad
    f_1 - \frac{f_2}{\iunit \omega + b} \in L^2((0,2\pi),\CSet),
  \]
  therefore we obtain
  \[
    |\omega| 
    \bigg(
    1 - \frac{1}{\omega_1^2 + b^2}  
    \bigg)
    \Vert u_1 \Vert_{L^2} 
    \leq
    \bigg\Vert f_1 - \frac{f_2}{\iunit \omega + b} \bigg\Vert_{L^2} 
    \leq
    \bigg(
      1 + \frac{1}{\sqrt{\omega_1^2 + b^2}} 
    \bigg)
    \Vert f \Vert_X 
  \]
  We conclude that there exist constants $K_2,\omega_2=1+\omega_1>0$ such that for all
  $\omega \in \RSet$ with $|\omega| > \max(\omega_1,\omega_2)$
  \[
    \Vert u_1 \Vert_{L^2} \leq \frac{K_2}{|\omega|} \Vert f \Vert_\cX.
  \]
  Taking the norm of the second equation in \cref{eq:FNHopfRes}, we arrive at the
  following bound, valid for all $\omega$ with $|\omega|>\max(\omega_1,\omega_2)$
  \[
    \begin{aligned}
    \Vert u_2 \Vert_{L^2} 
      & \leq \frac{1}{\sqrt{\omega^2 + b^2}} 
      \bigg( 
	1 + \frac{K_2}{|\omega_2|}
      \bigg) \Vert f \Vert_\cX
      = \frac{1}{|\omega|}  
      \frac{|\omega| + K_2}{\sqrt{\omega^2 + b^2}} \Vert f \Vert_\cX,
    \end{aligned}
  \]
  and we conclude that there exist $K_3,\omega_3 > 0$ such that, for all $\omega \in
  \RSet$ with $|\omega| > \max(\omega_1, \omega_2, \omega_3)$
  \[
    \Vert u_2 \Vert_{L^2} \leq \frac{K_3}{|\omega|} \Vert f \Vert_\cX.
  \]
  Using the bounds found for $\Vert u_1 \Vert_{L^2}$ and $\Vert u_2 \Vert_{L^2}$ we
  find, for all $\omega \in \RSet$ with $|\omega| >
  \max(\omega_1,\omega_2,\omega_3)$,
  \[
    \begin{aligned}
    \Vert (\iunit \omega - L)^{-1} f \Vert_{\cX} 
    & = 
      \big( \Vert u_1 \Vert_{L^2} + \Vert u_1 \Vert_{L^2} \big)^{1/2} \\
      & \leq \frac{\sqrt{K_2^2+K_3^2}}{|\omega|}
      \Vert f \Vert_{\cX}\,.
    \end{aligned}
  \]
  Therefore \cref{eq:intermediateBound2} holds and, hence, \cref{eq:FNHopfRes} hold
  with $\omega_0 = \max(\omega_1,\omega_2,\omega_3)$ and $K_0 = \sqrt{K_2^2+K_3^2}$.
\end{proof}

\subsubsection{System of 2-fast, 1-slow variables}
\label{subsec:FiniteDimResult}
The sections above show that the hypotheses of \cref{lem:HopfNormalForm} hold for the
FitzHugh-Nagumo reaction-diffusion system~\cref{eq:FNRev}. This leads to the
finite-dimensional Hopf normal form~\cref{eq:HopfNormalForm} for which the function
$G$ in the slow equation is the identity. We obtain the system
 \begin{equation}\label{eq:HopfNF}
    \begin{aligned}
      \dot A & = i \omega A + \alpha(B,\mu,\eps)A + \beta A |A|^2 + \delta \eps^2 +
      \textrm{h.o.t.}\\
      \dot B &= \eps
    \end{aligned}
  \end{equation}
  where $A$ and $B$ are one-dimensional complex and real variables, respectively.
One can then apply classical results on slow passage through a Hopf bifurcation to
this sytem. Such results were first unveiled by Shishkova~\cite{S73} and then proved
within a general framework by Neishtadt~\cite{neishtadt85,neishtadt87,neishtadt88};
see also~\cite{RB88,BER89,HE93} for an independent treatment. We therefore have a
theoretical local explanation for the numerical results shown in
\cref{fig:FHNSlowPassageHopf}. We note that it is straightforward to apply the theory
above to any reaction--diffusion system of PDEs, for which we expect to find
generically a slow-passage through Hopf bifurcations. Other types of boundary
conditions can also be easily included, by changing the underlying function spaces.

\subsection{Slow passage through Turing in a nonlocal reaction-diffusion equation}
\label{sec:slowPassageTuringAnalysis}
We have presented numerical results showing a slow-passage through a Turing
bifurcation for a system of reaction diffusion equations, and for a one-component
nonlocal reaction-diffusion equation model. In this section, we present in detail only
the nonlocal case, which is less well-studied. The computations for the other case
follow in a similar way, and we omit them here.

We recall the model under consideration,
\begin{equation}\label{eq:nonlocalRDRep}
  \begin{aligned}
    \partial_t u(x,t) & = d \partial_x^2 u(x,t) + v(t) u(x,t) - u(x,t) \int_\Omega
    w(x - y)
    u(y,t)\,dy,   \\
      \dot v & = \eps
  \end{aligned}
\end{equation}
for $(x,t) \in \Omega \times \RSet_{> 0}$, with $\Omega = \RSet/2l\ZSet$, which
implies periodic boundary conditions 
\[
  u(-l,t) = u(l,t), \qquad t \in \RSet_{> 0}.
\]
This model is obtained from $\cref{eq:nonlocalRD}$ by setting $v \mapsto v -b$, which
slightly simplifies the notation. We make some preliminary assumptions on the kernel
and on the existence of a homogeneous equilibrium.
\begin{hypothesis}\label{hyp:NonlocalDNTuring} 
  The interaction kernel $w \in L^1(\Omega)$ is an even $2l$-periodic function, with
  Fourier coefficients $(w_i)_{i\in \ZSet}$, $w_0 =1$. Further, there exist $u_*$,
  $v_*$, $d_* \in \RSet_{>0}$ and $n_* \in \ZSet_{\neq 0}$ such that $u(x,t) \equiv
  u_*$ is an equilibrium of \cref{eq:nonlocalRDRep} for $(v,d,\eps)=(v_*,d_*,0)$ with
  $u_*=v_*$, $d_* = -v_* w_{n_*} l^2/(n_*\pi)^2$.
\end{hypothesis}
As we shall see below, the requirement on $d_*$ guarantees that $u_*$ undergoes a
Turing-like bifurcation to a mode with wavelength $\pi/(l n_*)$. A change of
variables, similar to the one
used in other examples, leads to the system
\begin{equation}
  \begin{aligned}
    \partial_t u & = (d_* \partial_x^2 u - u_* W) u 
                   + (d \partial_x^2 + v) u + u_*v - uW u\\
      \dot v & = \eps
  \end{aligned}
\end{equation}
where we dropped the tildes and set $(Wu)(x) = \int_\Omega w(x-y) u(y)\,dy$. The
system above is in the form \cref{eq:slowFastInfRep} with

\[
  L = d_* \partial_x^2 - u_* W, \qquad R(u,v,\mu,\eps) = (\mu \partial_x^2 + v) u +
  u_*v - uW u, \quad G(u,v,\eps) = \eps
\]
where $\mu = d$.

\subsubsection{Choice of function spaces} In this problem, we choose $\cZ = 
H^2_\textrm{per}(\Omega)$ and $\cY = \cX = L^2_\textrm{per}(\Omega)$. We note that
$X=Y$ is a consequence of the fact that the linear operator in the original
problem~\cref{eq:nonlocalRDRep} contains a Laplacian and is
parameter dependent.

\subsubsection{Checking \cref{hyp:vectField}} Since $w \in L^1(\Omega)$ by
\cref{hyp:NonlocalDNTuring}, the Young's convolution theorem
\cite[Theorem 4.15]{brezis_functional:2010} gives
\[
  \Vert Wu \Vert_\cX \leq \Vert w \Vert_{L^1} \Vert u \Vert_\cX, 
	  \quad \textrm{for all $u \in \cX$},
	  \qquad
  \Vert Wu \Vert_\cZ \leq \Vert w \Vert_{L^1} \Vert u \Vert_\cZ,
          \quad \textrm{for all $u \in \cZ$}.
\]
We conclude that, since $\partial_x^2$ and $W$ are continuous operators from
$\cZ$ to $\cX$, then their linear combination $L$ is in $\mathcal{L}(\cZ,\cX)$. As
for the nonlinear function $R$, we note that the mapping $u \mapsto u Wu$, seen as a
nonlinear operator from $\cZ$ to $\cZ$, is the composition of $W$ and a product in
the Banach algebra $\cZ$. Since $W \in \mathcal{L}(\cZ)$, we conclude that $R \in
C^k(\cZ \times \RSet \times \RSet \times \RSet,\cX)$ for any integer $k$.

\subsubsection{Checking \cref{hyp:spectralDecomposition}}
We now study the spectrum of $L$ which is a closed operator in $\cX$ with domain
$\cZ$. Since $\cZ$ is compactly embedded in $\cX$ by the Sobolev embedding theorem (see
\cite{brezis_functional:2010}, Theorem~8.8), it follows that $L$ has a compact
resolvent. Thus, the spectrum of $L$ consists of at most a countable sequence of
isolated eigenvalues with finite algebraic multiplicity (see \cite[Theorem
III.6.29]{kato_perturbation:2005}) which can accumulate at $0$ or $\pm \infty$. For
the problem under consideration we can directly compute this sequence, and obtain
\[
  \sigma(L) = \{ -d_* (n \pi/l)^2 - u_* w_n \colon n \in \ZSet \}.
\]
Since $w \in L^1(\Omega)$, the Riemann--Lebesgue Lemma \cite[Chapter
6]{Serov:2017fz} implies $w_n \to 0$ as $n \to \infty$, hence
the spectrum accumulates at infinity. \cref{hyp:NonlocalDNTuring} guarantees that
\[
  \sigma_c(L) = \{0\} , \qquad \sigma_u(L) \cup \sigma_s(L) = 
  \{ -d_* (n \pi/l)^2 - u_* w_n \colon n \in \ZSet_{\neq n_*} \},
\]
and we have a spectral gap for $L$.

\subsubsection{Checking \cref{hyp:linearEquation}} Since $Y=X$, we proceed as in
  \cref{subsec:checkResolventHypHopf}, and we prove a result analogous to
  \cref{prop:boundResolventHopf}. We write $L = L_1 + L_2$, where $L_1 = d_* \partial_x^2$
  and $L_2 = - u_* W$. We note that $\sigma(L_1)$ is purely real, hence $i\omega \in
  \rho(L_1)$. We apply in sequence \cref{lemma:resolvent-laplacian,lemma:bounded-pert},
  and conclude that there exist positive constants $\omega_0, K_0$ such that
  $\Vert (\iunit \omega - L)^{-1} \Vert_{\mathcal{L}(\cX)} \leq K_0/|\omega|$
  for all $\omega \in \RSet$ with $|\omega| > \omega_0$, hence
  \cref{hyp:linearEquation} holds.

\subsubsection{Reduction to $1$-slow, $1$-fast system} 
We note that \cref{hyp:O2Equivariance} holds with $T_\phi$ and $S$ representing
translations by $\phi \in \RSet/2l\ZSet$ and reflections about the $x=0$ axis,
respectively, that is
\[
  (T_\phi u)(x) = u(x - \phi), \qquad (S u)(x) = u(-x), \qquad x,\phi \in
  \RSet/2l\ZSet.
\]
We can then apply \cref{lem:TuringNormalForm} and obtain, in a neighbourhood of the
origin, a $1$-slow, $1$-fast system 
\begin{equation} \label{eq:tempPitch}
  \begin{aligned}
  \dot A & = \alpha(B,\mu,\eps)A + \beta A |A|^2 + O\big( (|B| + |\mu| + \eps + |A|^2)^2 \big), \\
  \dot B &= \eps,
  \end{aligned}
\end{equation}
where $\alpha, \beta \in \RSet$.
The reduction above leads to studying an ODE with $1$ fast and $m$ slow variables.
We have not computed the coefficients $\alpha$ and $\beta$ explicitly, but the
numerical bifurcation analysis performed in \cref{sec:numericsNonlocalTuring}
indicates a pitchfork bifurcation for $\eps=0$ (a Turing bifurcation of the nonlocal
model). The results
in~\cite{KS01c} apply to \cref{eq:tempPitch} in this case: when $\eps \neq 0$ we
expect a slow-passage through the pitchfork bifurcation; this confirms what has been
found numerically in \cref{fig:nolocalRDSlowPassage}, namely the dynamics stays close
to the branch of homogeneous steady states, past the Turing bifurcation. Similar
considerations apply to the sub- and super-critical Turing bifurcations found in the
local model presented in \cref{sec:numericsSchnakenberg}.
%


\subsection{Spatio-temporal canards in a neural field model} 
\label{sec:slowPassageFoldAnalysis}
In this section we study
centre-manifold reductions for the nonlocal neural field problem \cref{eq:NF} where
$\Omega$ is a compact subset of $\RSet^d$. 
We introduce the integral operators
\[
  W \colon u \mapsto \int_\Omega w(\blank,y) u(y) \, d\rho(y),  
  \qquad
  Q \colon (u,v) \mapsto \int_\Omega \theta(u,v)(x) \, d\rho(x), 
\]
where the integrals are over $\Omega$, and the Nemytskii operator
$
  N \colon (u,v) \mapsto \theta(u(x),v),
$
and rewrite the original model as
\[
  \begin{aligned}
    \dot u             & = - u + WN(u,v_1), \\
    \dot v_1 & = \eps \big(  v_2 + c Q(u,v_1) \big), \\
    \dot v_2 & = \eps \big(  -v_1 + a + b Q(u,v_1) \big).
  \end{aligned}
\]
We make the following preliminary assumption:
 \begin{hypothesis}\label{hyp:NFSaddleNode}
  The firing rate function $\theta$ is in $C^k_b(\RSet^2)$, the space of bounded
  functions
  defined on $\RSet^2$ with continuous bounded derivative up to order $k$, for some
  $k \geq 2$. The synaptic kernel $w$ is in $C(\Omega^2)$. There exists
  $(u_*,v_{*1})
  \in C(\Omega) \times \RSet$ such that $u_* = WN(u_*,v_{*1})$. In addition, $1$ is a
  simple eigenvalue of $WD_uN(u_*,v_{*1})$, where 
  $D_u N(u_*,v_{*1}): u \mapsto \partial_u \theta(u_*,v_{*1}) u$ is regarded as an
  operator from $C(\Omega)$ to itself.
\end{hypothesis}

The setting of this example is different from the previous ones. Firstly, we note
that in this case $u_*$ is not a homogeneous equilibrium of the model, and indeed
canards were found in \cite{avitabile2017spatiotemporal} close to saddle-node
bifurcations of heterogeneous equilibria. Secondly, in this problem, we aim to
provide an example of evolution equation on a Banach space as opposed to a Hilbert
space (which is why we demand $u_* \in C(\Omega)$). This functional setting is more
challenging, because we can not use resolvent estimates as in
\cref{prop:FNResolvent,prop:boundResolventHopf}, and instead
\cref{hyp:linearEquation} must be checked directly. A Hilbert-space setting is
however possible, and relies on results in \cite{mielke1988,Haragus2010}.

After setting $u = u_* + \tilde u$, $v_1 = v_{*1} + \tilde v_1$, and dropping tildes,
we obtain
\begin{equation}\label{eq:NFBanach}
  \begin{aligned}
    \dot u             & = - u + WD_uN(u_*,v_{*1})u + R(u,v_1)\\
  \dot v_1 & = \eps \big(  v_2 + c Q(u_* + u ,v_{*1}+ v_1) \big) \\
  \dot v_2 & = \eps \big(  - v_{*1} -v_1  + a + b Q(u_* + u ,v_{*1}+ v_1) \big)
  \end{aligned}
\end{equation}
which fits in our framework with $\mu = (a,b,c)$,
\begin{equation}\label{eq:LOpNF}
    L\colon u \mapsto -u + W D_u N(u_*,v_{1*}) u,
\end{equation}
and nonlinear terms $R,G$ given by
\[
  \begin{aligned}
   R(u,v_1) &= v_1 \int_\Omega w(x,y) \partial_{v_1} 
  \theta (u_*(y),v_{*1}) \, dy + o( |u| + |v_1| ) \\
   G(u,v,\mu) & = 
  \big(  v_2 + \mu_3 Q(u_* + u ,v_{*1}+ v_1) ,
    - v_{*1} -v_1  + \mu_1 + \mu_2 Q(u_* + u ,v_{*1}+ v_1) \big)
  \end{aligned}
\]

\subsubsection{Function spaces and \cref{hyp:vectField}} For this example we choose
$Z = Y = X = C(\Omega)$. For all $u \in Z$ we have, owing to \cref{hyp:NFSaddleNode},
  \[
    \Vert Lu \Vert_\cX = \Vert Lu \Vert_\infty 
    \leq 
    \big(
      1 + \meas(\Omega) \Vert w \Vert_\infty \Vert \partial_u \theta \Vert_\infty 
    \big)
    \Vert u \Vert_\infty := K \Vert u \Vert_\cZ,
  \]
  where the infinity norms are on $C(\Omega)$ or $C(\Omega^2)$, as appropriate.
  Therefore $L \in \mathcal{L}(Z,X)$. The nonlinear terms $R$ satisfy
  \cref{hyp:vectField}.

\subsubsection{Checking \cref{hyp:spectralDecomposition}} 
\label{sec:hypSpecNF}
We first recall the following result:
\begin{lemma}
  If \cref{hyp:NFSaddleNode} holds the operator $WD_uN\in\mathcal L(\cX)$ is compact.	
\end{lemma}
\begin{proof}
As $\bar\Omega$ is compact, the Weierstrass theorem provides a sequence of polynomials $(w_k)_k$
which tends to $w$ in $C(\bar\Omega^2)$, and the same is true for the
associated integral operators $(W_k)_k$ in $\mathcal L(\cX)$. These operators have
finite dimensional range and are thus compact operators. The set of compact
operators is closed in $\mathcal L(X)$ (see \cite{brezis_functional:2010}) hence leading to the conclusion that $W$ is
compact. As $D_uN$ is a continuous operator on $X$, $WD_uN$ is compact.
\end{proof}

It follows that the spectrum $\sigma(WD_uN)$ of $WD_uN$ contains $0$, and
$\sigma(WD_uN)\setminus\{0\}$ consists at most of a countable sequence of isolated
eigenvalues with finite algebraic multiplicity which may accumulate at $0$. Since $L
= -\id_X + WD_u N$ and \cref{hyp:NFSaddleNode} holds,
\cref{hyp:spectralDecomposition}(i) is verified, and there exists a spectral gap
$\gamma>0$ which may be chosen such that \cref{hyp:spectralDecomposition}(ii) is
satisfied. In addition, the structure of the spectrum of $WD_uN$ implies that
$\sigma_u(L)$ is finite and $\cX_u$ is finite dimensional. 
%

\subsubsection{Checking \cref{hyp:linearEquation}}

As anticipated above, since we are
dealing with a quasilinear formulation ($\cY = \cX$) on a Banach space, we can not make
use of resolvent estimates and we must verify \cref{hyp:linearEquation} directly (see
also \cite{iooss_travelling:2000} for a similar strategy). 
In the following, we write $T(t)=e^{tL}\in\mathcal L(\cX)$ the semigroup generated by
$L$. We initially derive the following useful estimates:
\begin{proposition} \label{prop:expEstimates}
  Assume \cref{hyp:NFSaddleNode}, and let $L$ be the operator defined in
  \cref{eq:LOpNF}, and let $P_u$, $P_s$ be the associated projectors on $X_u$, $X_s$,
  respectively.
    There exists a constant $M>0$ such that
\begin{align}
& \norm{T(t)P_{u}} \leq Me^{\beta_u t}
&&
\text{for all $t \leq 0$, $\beta_u \in \Big(0,\min_{\lambda \in \sigma_u} \real \lambda\Big)$},
\label{eq:unstabBound} \\
& \norm{T(t)P_{s}} \leq Me^{-\beta_s t}
&&\text{for all $t \geq 0$, $\beta_s \in \Big(0,-\sup_{\lambda \in \sigma_s} \real
\lambda\Big)$}.
\label{eq:stabBound}
\end{align}
\end{proposition}
\begin{proof}
  The space $X_u$ is finite dimensional(see \cref{sec:hypSpecNF}), hence
  $T(t)P_{u}$ has finite-dimensional range, and a Dunford decomposition yields
  \cref{eq:unstabBound}. The operator ${T(t)P_{s}}$ is bounded on $\cX$. As a consequence, the spectral
  mapping theorem \cite[Section VII.3.6, Theorem 11]{dunford_linear:1988} yields
  $\sigma\left(T(t)P_s\right)
\setminus\{0\}= e^{t\sigma_s(L)}$. Finally, an application of the Gelfand spectral
radius theorem \cite{dunford_linear:1988} yields \cref{eq:stabBound}. Note that \cref{hyp:NFSaddleNode} provides the existence of $\beta_u,\beta_s$.
\end{proof}

\begin{proposition}\label{prop:NFHypLinearEquation}
  Assume the hypotheses of \cref{prop:expEstimates}.
Then
\cref{hyp:linearEquation} holds with
\[
(K_{su}f)(t) = -\int_t^\infty T(t-\tau)P_u f(\tau)\,d\tau 
+\int_{-\infty}^tT(t-\tau)P_sf(\tau)\,d\tau\,.
\]
\end{proposition}
\begin{proof}
  \cref{hyp:linearEquation} is stated in terms of the constant $\gamma>0$, fixed as
  in \cref{hyp:spectralDecomposition}(ii). The constant can be chosen to be
  $\gamma = \min(\beta_s,\beta_u)/2$, where $\beta_u,\beta_s$ are in the range
  specified in \cref{prop:expEstimates}.
  Henceforth we fix $\eta \in [0,\gamma]$, $f \in C_\eta(\RSet,Y_{su})$, let
  $u = K_{su}f$, and we prove that \cref{hyp:linearEquation} holds for $u$, using 4
  steps.

\emph{Claim 1: $u$ is exponentially bounded.} Owing to the standard properties of
the projectors $P_s,P_u$, we have $u(t)\in \cZ_{su}$ for all $t\in\mathbb R$.
We set $u = u_s + u_u$ with 
\[
  u_s(t)  =\int_{-\infty}^tT(t-\tau)P_sf(\tau)\,d\tau, 
  \qquad
  u_u(t) = -\int_t^\infty T(t-\tau)P_uf(\tau)\,d\tau,
\]
and use \cref{eq:stabBound} to derive the following bound
\[
\begin{aligned}
  \norm{u_u(t)}_\cZ 
 & \leq M\norm{f}_\eta \int_t^\infty e^{\beta_u(t-\tau)+\eta|\tau|}\,d\tau \\
 & \leq M\norm{f}_\eta e^{\eta|t|} \int_0^\infty e^{(\tau-\beta_u)\tau}\,d \tau=
 \frac{M}{\beta_u-\eta}\norm{f}_\eta e^{\eta |t|}.
\end{aligned}
\]
A similar bound can be found for $u_s$, and we obtain
\[
  \norm{u(t)}_\cZ\leq \bigg( \frac{1}{\beta_u-\eta}+\frac{1}{\beta_s-\eta} \bigg)
  M\norm{f}_\eta e^{\eta|t|}.
\]

\emph{Claim 2: $u$ is in $C_\eta(\mathbb R,\cZ_{su})$.}
The above estimate, the continuity of $f \colon \RSet \to \cX$ and the dominated
convergence theorem imply that $u \colon \mathbb \RSet \to \cX$ is continuous and $u\in
C_\eta(\mathbb R,\cZ_{su})$. In passing we note that the estimate above also implies
$\norm{K_{su}}\leq \Pi(\eta)=M/(\beta_u-\eta)+M/(\beta_s-\eta)$.

\emph{Claim 3: $u$ is differentiable in $\cX$ and solves \eqref{eq:linProbFastSubs}.}
We treat the term $u_s(t)$ for some fixed $t\in\mathbb R$, as the result for $u_u(t)$
follows in a similar way. Using a Taylor expansion of the exponential
$T_s(\tau)=e^{L_s\tau}$ we obtain, for sufficiently small $\eps$,
\[
  \begin{aligned}
u_s(t+\eps)-u_s(t) 
& =
\int_{-\infty}^t [T_s(t-\tau + \eps) - T_s(t-\tau)] f(\tau)\, d\tau
+
\int_{t}^{t+\eps} T_s(t-\tau + \eps) f(\tau)\, d\tau
    \\
    & =\eps L_s \int_{-\infty}^t T_s(t-\tau)f(\tau)\,d\tau 
    + T_s(\eps) \int_{t}^{t+\eps} T_s(t-\tau) f(\tau)\, d\tau + O(\eps^2)
	\\
	& =\eps L_s \int_{-\infty}^t T_s(t-\tau)f(\tau)\,d\tau 
	+ \eps P_s f(t) + O(\eps^2).
  \end{aligned}
\]
Hence, $u_s$ is differentiable at $t$ and $\dot u_s(t) = P_sf(t)+L_su_s(t)$. It
follows that $u_s$ is $C^1$ in $\cX$. The claim is proved using a similar argument for
$u_u$.

\emph{Claim 4: $u$ is unique.} Assume there are two solutions $u,\tilde u$ to
\eqref{eq:linProbFastSubs}. Then $\nu = u - \tilde u$ solves $\dot \nu =
L_{su}\nu$ and belongs to $C_\eta(\mathbb R,\cX_{su})$. This implies $\nu(0)\in
X_s\cap X_u=\{0\}$, hence $\nu(t) \equiv 0$ and the uniqueness of $u$.
\end{proof}

\subsubsection{Reduction to 1-fast 2-slow system} 
We now apply \cref{lem:FoldNormalForm} to \cref{eq:NFBanach}. We note that
$\gamma(\mu,\eps) = 0$ because $R$ is independent of $(\mu,\eps)$; in addition, $R$
does not depend on $v_2$, hence $\alpha(B)$ and higher order terms depend solely on
$B_1$.
In a neighbourhood of the origin we obtain the reduction
  \begin{equation}\label{eq:NFReduced}
  \begin{aligned}
  \dot A & = \alpha B_1 + \beta A^2 
  + \omega(A,B_1,\mu,\eps) \\
\dot B_1 & = \eps 
\big(
  B_2 + \mu_3 H(A,B_1,B_2,\mu,\eps)
\big) \\
\dot B_2 & = \eps 
\big(
  -v_{*1} - B_1 + \mu_1 + \mu_2 H(A,B_1,B_2,\mu,\eps)
\big) \,,\\
  \end{aligned}
\end{equation}
where $H \colon \RSet^3 \times \RSet^p \times \RSet \to \RSet$ is defined
by
\begin{equation}\label{eq:HDef}
  H(A,B_1,B_2,\mu,\eps) = Q(u_* + A \zeta + \Psi(A,(B_1,B_2),\mu,\eps),v_{*1} + B_1),
\end{equation}
and
\[
  \omega(A,B_1,\mu,\eps) = O\big(A^2(|B_1|+|\mu|+|\eps|)+(|B_1|+|\mu|+|\eps|)^2\big).
\]

\begin{lemma}\label{lem:NFFoldedSing}
  Assume $\alpha, \beta \neq 0$, and let $O_\mu$ be defined as in
  \cref{lem:FoldNormalForm}. There exists an open subset $U_\mu \subset O_\mu$ and a
  function
  $\xi \colon U_\mu \to \RSet$ such that system \eqref{eq:NFReduced} admits a
  folded singularity at $(A,B_1,B_2,\mu) = (0,0,\xi(\mu),\mu)$ for all $\mu \in
  U_\mu$.  The
  folded singularity is
  \begin{itemize}
    \item a folded saddle if $J_{12}J_{21}>0$,
    \item a folded node if $J_{11} < 0$ and $J_{11}^2 + 4 J_{12}J_{21} > 0$,
    \item a folded saddle-node if $J_{12}J_{21} = 0$,
  \end{itemize}
  where
  \[
    \begin{aligned}
    J_{11} & = \mu_3 \beta/\alpha \big[
      \partial_A H(0,0,\xi(\mu),\mu,0) + \partial_{B_1}
      H(0,0,\xi(\mu),\mu,0)  \partial_A \eta(0,\mu)
    \big],\\
    J_{12} & = \beta/\alpha 
    \big[
      1 + \mu_3 \partial_{B_2} H (0,0,\xi(\mu),\mu,0) 
    \big], \\
    J_{21} & = -2( -v_{*1} + \mu_1 + \mu_2 H(0.0,\xi(\mu),\mu,0).
    \end{aligned}
  \]
\end{lemma}
\begin{proof}
  By setting $\tilde t = t/\beta$, $\tilde \eps = \eps/\beta$, $\kappa =
  \beta/\alpha$ (and dropping the tilde in $\eps$) we cast system \cref{eq:NFReduced}
  as
\begin{equation}\label{eq:NFReducedScaled}
  \begin{aligned}
  \dot A & = \kappa B_1 + A^2 +\omega(A,B_1,\mu,\eps), \\
  \dot B_1 & = \eps
      \big(
	B_2 + \mu_3 H(A,B_1,B_2,\mu,\eps)
      \big), \\
      \dot B_2 & = \eps 
      \big(
	-v_{*1} - B_1 + \mu_1 + \mu_2 H(A,B_1,B_2,\mu,\eps)
    \big). \\
  \end{aligned}
\end{equation}
After passing to the slow time $\tau = \eps t$ in \cref{eq:NFReducedScaled} we
obtain, at $\eps = 0$, the slow subystem
\begin{equation}\label{eq:NFDAE}
  \begin{aligned}
       0 & = \kappa B_1 + A^2 +\omega(A,B_1,\mu,0), \\
    B'_1 & = 
      \big(
	B_2 + \mu_3 H(A,B_1,B_2,\mu,0)
      \big), \\
      B'_2 & = 
      \big(
	-v_{*1} - B_1 + \mu_1 + \mu_2 H(A,B_1,B_2,\mu,0)
    \big). \\
  \end{aligned}
\end{equation}
The critical manifold is a graph over $A$ in the neighbourhood of the origin; the
Implicit Function Theorem guarantees the existence of a subset $V_A \times V_\mu$ of
$O_A \times O_\mu$, and a unique function $\eta: V_A \times V_\mu \to
\RSet$ such that
\[
  0 = \kappa \eta(A,\mu) + A^2 +\omega(A,\eta(A,\mu),\mu,0) 
  \qquad
  (A,\mu) \in V_A \times V_\mu.
\]
The desingularised reduced system associated to \cref{eq:NFDAE} is given by
  \begin{equation}\label{eq:NFDRS}
  \begin{aligned}
    \dot A & = 
        \big[\kappa + \partial_{B_1} \omega(A,\eta(A,\mu),\mu,0)\big]
	\big[B_2 + \mu_3 H(A,\eta(A,\mu),B_2,\mu,0) \big] \\
    \dot B_2 & = \big[-2A - \partial_{A} \omega(A,\eta(A,\mu),\mu,0)\big] \big[ -v_{*1} - B_1 +
    \mu_1 + \mu_2 H(A,\eta(A,\mu),B_2,\mu,0)\big].
  \end{aligned}
\end{equation}
Since by the Implicit Function Theorem
$\kappa + \partial_{B_1} \omega(A,\eta(A,\mu),\mu,0)$ does not vanish in $V_A \times
V_\mu$, then a folded singularity of \cref{eq:NFReduced} is an equilibrium of
\cref{eq:NFDRS} satisfying
  \begin{align*}
    0 & =  B_2 + \mu_3
    H(A,\eta(A,\mu),B_2,\mu,0), \\
    0 & = -2A - \partial_{A} \omega(A,\eta(A,\mu),\mu,0).
  \end{align*}
  The second equation holds for $A=0$. A further application of the Implicit Function
  Theorem to the first equation with $A =0$ then yields the existence of $U_\mu
  \subset V_\mu$ a function
  $\xi \colon U_\mu \to \RSet$ such that $(A,B_1,B_2)= (0,0,\xi(\mu))$ is a folded
  singularity. The statement follows from
the fact that the Jacobian of \cref{eq:NFDRS} at the folded singularity is the matrix
\[
  \begin{bmatrix}
    J_{11} & J_{12} \\
    J_{21} & 0
  \end{bmatrix}.
\]
\end{proof}

The previous lemma implies that solutions to the original neural field model near
saddle-node bifurcation points of the fast subsystem, which are patterned states,
are expected to display canard segments. The existence of these structures was
predicted analytically in \cite{avitabile2017spatiotemporal} in the case of Heaviside
firing rate $\theta$ and is therefore valid for generic smooth, bounded firing rates.
Simulations have been reproduced from that paper in
\cref{fig:NeuralFieldRing,fig:NeuralFieldSphere}.  We note that, while
\cref{lem:NFFoldedSing} does not directly apply when $\Omega$ is a ring or a sphere,
similar dynamics is observable on Neural Field models with a heterogeneous external
input, which are not equivariant with respect to the action of a Lie group, and to which
\cref{lem:NFFoldedSing} applies.

\subsection{Delay-differential equation}
\label{sec:AnalysisDDE}
We now present results for the DDE \cref{eq:DDE}. The
treatment of this case is separate to the others, as it requires different technical
tools. We provide pointers to specialised literature on the topic further below in
the section.

We denote by $u_*$ an equilibrium of the fast subsystem associated to \eqref{eq:DDE},
set $u=u_*+x$, and obtain
\begin{equation}
\begin{aligned}
\dot x(t) &= v x(t) - x(t-\tau) +u_*^3-(u_*+x(t))^3,
	  && t\geq 0\\
x(t) &= \phi(t),&& t\in[-\tau,0].
\end{aligned}
\label{eq:DDE2}
\end{equation}

The analysis of delay differential equations is technical and
relies on tools from functional analysis and semigroup theory
\cite{engel2000}. In this section, we aim to provide a self-contained treatment for
the analysis of \eqref{eq:DDE2}, and we refer the reader to the textbooks
\cite{hale_introduction_2014, diekmann_delay_1995} for an exhaustive treatment of the
subject.

To make sense of the right hand side of \eqref{eq:DDE2}, $x$ must be known on a time
interval of length $\tau$, typically a \textit{history} time interval $[t-\tau,t]$. This implies that the state space of \eqref{eq:DDE2} is
infinite dimensional and included in some function space from $[-\tau,0]$ to $\mathbb
R$. The choice of this function space affects the analysis of \eqref{eq:DDE2}. For
example, if we were to choose $\cX=C^0([-\tau,0],\RSet)$, it
would appear that the operator $L$ in \eqref{eq:fastSubsystemInf} encodes the delay
differential equation in its domain $\cZ$
(see \cite{hale_introduction_2014} and Remark~\ref{rmq:dde}). We therefore would not
apply the centre-manifold results presented in the previous sections, as the space
$\cZ$ would depend on the parameter $v$.

The standard formalism aimed to tackle this difficulty is that of the
\textit{sun-star} calculus \cite{diekmann_center_1991,diekmann_delay_1995} which
seeks solutions in a larger state space of less regular functions. This theory can
also be applied to centre-manifold reductions \cite{diekmann_delay_1995}. Here we
present a self-contained example that does not require the additional technical
knowledge of sun-star calculus. We proceed as follows:
\begin{enumerate}
  \item We select the state spaces $\cZ,\cY,\cX$ as suggested by the sun-star
    theory and check that they are suitable for our problem (see also
    \cite{batkai_semigroups_2005} for a way to bypass the use of sun-star calculus).
  \item We do not apply the center manifold reduction exposed in
    \cite{diekmann_center_1991,
    diekmann_delay_1995} because that approach looks for integral solutions to the
    problem and thus solves \eqref{eq:linProbFastSubs} in $C_\eta(\RSet,\cX)$ instead
    of $C_\eta(\RSet,\cZ)$, which is required by our formalism. It is possible to
    adapt the
    proof in \cite{diekmann_center_1991, diekmann_delay_1995} to fit our framework,
    but we choose a direct and self-contained approach. We highlight that we
    will exploit (in Lemma~\ref{lemma:hyp5.3-dde}) the fact that the
    nonlinearity $R$ has a finite-dimensional range, as it is done in the sun-star
    references given above.
\end{enumerate}

Before proceeding we introduce, for any $t\geq 0$ the \textit{history function}
$w_t:[-\tau,0]\to\cX$ of $w:[-\tau,\infty)\to\cX$ defined by
$w_t(\theta):=w(t+\theta)$.

\subsubsection{Function spaces and \cref{hyp:vectField}}
A first step in the presentation is to rewrite \eqref{eq:DDE2} as a Cauchy problem.
We chose a Hilbert space setting with $\cX= \cY = \RSet\times L^2((-\tau,0),\RSet)$ and we denote
by $\pi_1$, $\pi_2$ the canonical projections from $\cX$ onto $\RSet$ and
$L^2((-\tau,0),\RSet)$, respectively. We set $\cZ = 
\{u\in \RSet\times W^{1,2}((-\tau,0),\RSet)\ | \ (\pi_2u)(0) = \pi_1u\}$ and note that
$\cZ$ is continuously embedded in $\cX$. This allows us to define the following
operators (see \cite{batkai_semigroups_2005}):
\begin{equation}\label{eq:DDELR}
L=\begin{pmatrix}
v & \Phi \\
0 & \frac{d}{d\theta}
\end{pmatrix}\in\mathcal L(Z,X),
\quad R(u) = 
\begin{pmatrix}
u_*^3-(u_*+\pi_1(u))^3 \\
0
\end{pmatrix}
\end{equation}
where $\Phi=u_2\to -u_2(-\tau)\in\mathcal L(\mathcal\rm
W^{1,2}(-\tau,0;\RSet),\RSet)$. Note that $R(u)\notin \cZ$ for $u\in\cZ$. However,
$R\in C^\infty(\cZ,\cX)$ because $\cZ$ is a Banach algebra. 

In conclusion, we consider the Cauchy problem
\[\dot u = Lu +R(u)\]
with initial condition in $\cZ$.
Although it may appear mysterious at first, one can check, using the first
component, that \eqref{eq:DDE2} is indeed encoded by the above Cauchy problem and
this provides \cref{hyp:vectField}.

\begin{remark}\label{rmq:dde}
  A possible alternative choice \cite{hale_introduction_2014,diekmann_delay_1995} is to
  define 
  \[
  \cX=C^0([-\tau,0],\RSet),
  \quad
  \cZ = \bigg\{
    \phi\in C^1([-\tau,0],\RSet)\ | \  \frac{d}{d\theta}\phi(0) = v\phi(0) -
  \phi(-\tau) 
\bigg\},
  \]
  and then
  \[
  L:\cZ\to \cX, \qquad u \mapsto \frac{d}{d\theta}u. 
  \]
  With this choice $\cZ$ depends on the parameter $v$, which is not suitable for
  our analysis because the spaces $\cZ,\cY,\cX$ cannot depend on a varying parameter.
  One then can use the sun-star formalism
  \cite{diekmann_center_1991,diekmann_delay_1995,engel2000}, akin to a double dual
  operation, to rewrite the problem in the space $\cX = \mathbb R\times
  L^\infty((-\tau,0),\mathbb R)$ obtaining a form similar to \cref{eq:DDELR}.
\end{remark}

\subsubsection{Checking \cref{hyp:spectralDecomposition}} We now check hypotheses on
the spectrum of $L$ in \cref{eq:DDELR}, which is characterised in the following lemma:
\begin{lemma}\label{lemma:DDEspectrum}
 Fix $\tau>0 $, $v
 \in\RSet$. The spectrum of $L$ is
 composed solely of eigenvalues and is given by $\sigma(L) = \{
 \lambda \in \CSet \colon v - e^{-\lambda \tau} -\lambda = 0\}$. An eigenvector associated to the eigenvalue $\lambda \in \sigma(L)$ is $(1,e^{\lambda\blank})\in\cZ$.
\end{lemma}
\begin{proof}
In this proof, we write $L^2 = L^{2}((-\tau,0),\RSet)$ and $W^{1,2} = W^{1,2}((-\tau,0),\RSet)$ to simplify notations.	
Let $E=\{\lambda\in\CSet:\ v-e^{-\lambda\tau}-\lambda\neq0\}$. We prove that
$E\supseteq\rho(L)$ by showing that $\CSet\setminus E\subset\sigma(L)$. For each
$\lambda$ solution of $v-e^{-\lambda\tau}-\lambda=0$, $(1,e^{\lambda\blank})\in\cZ$
is an eigenvector showing that $\lambda-L$ is not invertible.

Conversely, we consider $\lambda\in E$. We show that the equation $Lu=\lambda u +
w$ has a unique solution $u\in\cZ$ for any $w\in \cX$. Writing $u=(u_1,u_2)$, one
finds 
\[
 \frac{d}{d\theta}u_2=\lambda u_2+w_2, \qquad
u_2(\theta) = e^{\lambda\theta}u_1+\int_0^\theta
e^{\lambda(\theta-s)}w_2(s)ds.
\]
The implies that $u_2$ is of Sobolev regularity. One also has 
$v u_1-u_2(-\tau) = \lambda u_1+w_1$ which gives
\[(v-e^{-\lambda\tau}-\lambda)u_1 = w_1+\int_0^{-\tau} e^{\lambda(-\tau-s)}w_2(s)ds.\]
This equation has a unique solution $u_1\in\RSet$ because $\lambda\in E$ and this
provides the unique solution to $Lu=\lambda u + w$. Finally, $\Vert u \Vert_{\cZ}=
|u_1|+\Vert u_2 \Vert_{W^{1,2}}$. From the Cauchy-Schwartz inequality
\[
|u_1|\leq \frac{1}{|v-e^{-\lambda\tau}-\lambda|}\left(|w_1|+\Vert e^{\real\lambda\blank}
\Vert_{L^{2}}\Vert w_2
\Vert_{L^{2}}\right)=O\left(\Vert w \Vert_{\cX}\right).
\]
From the definition of $u_2$, we find
\[
\Vert u_2 \Vert_{L^{2}}\leq \Vert e^{\lambda\blank} \Vert_\infty\left(|u_1| + \sqrt\tau\Vert w_2 \Vert_{L^{2}}  \right)=O\left(\Vert w \Vert_{\cX}\right).
\]
Finally, 
\[
\bigg \Vert \frac{d}{d\theta}u_2 \bigg\Vert_{L^{2}}\leq |\lambda|\cdot\Vert u_2 \Vert_{L^{2}}+\Vert w_2 \Vert_{L^{2}}=O\left(\Vert w \Vert_{\cX}\right).
\]
This implies that $\Vert u \Vert_{W^{1,2}}=O\left(\Vert w \Vert_{\cX}\right)$.
We have shown that 
\[ \sigma(L) = \{\lambda\in\CSet:\ v-e^{-\lambda\tau}-\lambda = 0\}.\]
We note that the spectrum is composed of eigenvalues $\lambda$ because
the vector $(1,e^{\lambda\blank})\in\cZ$ is an associated eigenvector. 
\end{proof}

The spectrum of $L$ is composed of the
zeros of an holomorphic function, hence the (at most countable) spectral elements are isolated. Therefore
\cref{hyp:spectralDecomposition} holds and $\sigma_u$ is finite. As a consequence
$\cX_u$ is finite dimensional. Note that the Dunford projectors $P_u,P_s$
are easily expressed using the expression of the eigenvectors and the scalar product on $\cX$.

\begin{remark}
Using the different branches $(W_k)_{k\in\mathbb Z}$ of the Lambert function $W$ (see \cite{corless_lambertw_1996}) which satisfies $W(z)e^{W(z)}=z,\ z\in\mathbb C$, it is possible to compute all the eigenvalues
\[
\lambda_k = v+\frac1\tau W_k\left(-\tau e^{-\tau v}\right),
 k\in\mathbb Z.
\]
\end{remark}

\subsubsection{Checking \cref{hyp:linearEquation}}
We now check the last hypothesis required to have a center manifold, namely
\Cref{hyp:linearEquation} which is more challenging than the previous steps. The case
we study is a nontrivial example for which the sufficient conditions stated in ~\cite[Section
2.2.3]{Haragus2010} do not hold, whereas \Cref{hyp:linearEquation} is satisfied. As
for the neural field example, we check directly that the solution to
\eqref{eq:linProbFastSubs} is 
\[
(K_{su}f)(t) = -\int_t^\infty T(t-r)P_uf(r)dr+\int_{-\infty}^tT(t-r)P_sf(r)dr
\]
where $(T(t))$ is the (strongly continuous) semigroup of solutions generated by $L$. This semigroup can be directly expressed using the method of steps. We first give the expression of $(T(t))$ which allows to properly define $K_{su}$.

\begin{lemma}\label{lemma:ddeT}
Let $\vecd{x}{\phi}\in\cX$. For each $t\in[0,\tau]$ we have
\begin{equation}
  \begin{aligned}
    \pi_1\left(T(t)\vecd{x}{\phi}\right) & = e^{vt}x-\int_0^te^{v(t-s)}\phi(s-\tau)ds:=u_1(t)\\
    \pi_2\left(T(t)\vecd{x}{\phi}\right) & = (u_1)_t
  \end{aligned}
\end{equation} 
where $(u_1)_t$ is the history function of $u_1$. In addition,
$\range(T(\tau))\subset \cZ$.
\end{lemma}
\begin{proof}
Let us compute the linear flow $u(t)=T(t)\vecd{x}{\phi}$ for $\vecd{x}{\phi}\in\cX$.
We write $u_i=\pi_i(u)$, for $i=1,2$.
We start with $\vecd{x}{\phi}\in D(L)=\cZ$. Proposition~3.9 in \cite{batkai_semigroups_2005} shows that $u_1$ is solution of 
\begin{equation*}
\begin{aligned}
  \dot u_1(t) &= v u_1(t) + \Phi ((u_1)_t), & & t\geq 0\\
  u_1(0)&=x, & &\\
  u_1(\theta)& = \phi(\theta), & & t\in [-\tau,0)
\end{aligned}
\end{equation*}
and $u_2(t) = (u_1)_t$.
We have $\dot u_1(t) = v u_1(t) - u_1(t-\tau)$ and for $t\in[0,\tau]$, we get 
\begin{equation}\label{eq:ddesolution}
u_1(t) = e^{vt}x-\int_0^te^{v(t-s)}\phi(s-\tau)ds.
\end{equation}
We thus have the expression of $u(t)=T(t)\vecd{x}{\phi}$ for $t\in[0,\tau]$ on $\cZ$.
Further, as $\cZ$ is dense in $\cX$, we can consider a sequence such that
$\cZ\ni\vecd{x_n}{\phi_n}\to \vecd{x}{\phi}\in\cX$. The above analytical expression
shows that the $\pi_1(T(t)\vecd{x_n}{\phi_n})$ converges to the right hand side of
\eqref{eq:ddesolution} in $\cX$, namely 
\[
\pi_1(T(t)\vecd{x}{\phi}) =
e^{vt}x-\int_0^te^{v(t-s)}\phi(s-\tau)ds. 
\]
Proposition~3.11 in
\cite{batkai_semigroups_2005} shows that the second component
$\pi_2(T(t)\vecd{x}{\phi})$ is actually $(u_1)_t$. We thus have found the expression
of $T(t)$ for $t\in[0,\tau]$ in the whole space $\cX$. 
The last statement is straightforward to check, using an approach similar to the proof of Lemma~\ref{lemma:DDEspectrum}.
\end{proof}

Let us consider a general $t=n\tau+s$ with $s\in[0,\tau)$, the semigroup property
gives $T(t) = T(\tau)^nT(s)$ and we thus have the expression of $T(t)$ for any $t\geq
0$. We also assume that Proposition~\ref{prop:expEstimates} holds for the above
semigroup, this is usually proved \cite{veltz_center_2013,hale_introduction_2014} by
showing that $T$ is eventually norm continuous. Therefore, we have everything at hand
to define $K_{su}f$ as above.

In order to check \cref{hyp:linearEquation}, we need to:
\begin{enumerate}
\item show that $K_{su}\in \mathcal L(C_\eta(\RSet,\cX_{su}),C_\eta(\mathbb R,\cZ_{su}))$, 
\item show that $K_{su}f\in C^1(\RSet, \cX)$,
\item show that equality \eqref{eq:linProbFastSubs} holds in $\cX$ for all $t\in I$.	
\end{enumerate}

We will not complete all the steps above, as this is laborious and
parallels the proof of \cite[Proposition~C4]{veltz_center_2013}. Instead, we
underline some salient points which are different from the previous examples, the
most notable one being the gain of regularity of the solution, that is, $K_{su}f$
belongs to $\cZ$, and is differentiable in $\cX$. 

An additional technical point which proves useful is the following. From the proof of the center manifold theorem in \cite{Haragus2010}, it can be noted
that the linear operator $K_{su}$ is always applied to vectors such as $P_{su}R(u)$,
$P_{su}DR(u),\cdots$. Given the particular form of these vectors
$P_{su}\begin{pmatrix}\alpha\\0\end{pmatrix}$ for some $\alpha\in\mathbb R$ stemming
from \eqref{eq:DDELR}, we only have to solve \eqref{eq:linProbFastSubs} for
functions $f$ which have the same shape, \textit{i.e.} that belong to the specific linear subspace $P_{su}(\RSet\times\{0\})$ of $\cX$.

We would wish to proceed as in the proof of
Proposition~\ref{prop:NFHypLinearEquation}, but we must adapt it
because the spaces $\cZ,\cX$ are different in the present case and this requires to
show the gain of regularity. This gain of regularity from $\cX$ to $\cZ$ is provided
below. For the continuity of $K_{su}$ from $C_\eta(\mathbb R,\cX_{su})$ in
$C_\eta(\mathbb R,\cZ_{su})$, we refer to \cite[Proposition~C4]{veltz_center_2013}. 

\begin{lemma}\label{lemma:hyp5.3-dde}
Let $f\in C_\eta(\RSet,\RSet)$ and consider the stable component of $K_{su}$:
  \[
  u(t) = \int_{-\infty}^tT(t-r)P_s\vecd{f(r)}{0}dr.
  \]
Then for all $t \in \RSet$ $u(t)\in\cZ$, $u$ is differentiable in $\cX$ and $\dot
u(t)=Lu(t)+P_s\vecd{f(t)}{0}$.
\end{lemma}
\begin{proof}
\cref{prop:expEstimates} allows to give a meaning to the expression of $u$ by
showing that it is well defined and finite. Here we first show that $u(t)\in\cZ$.
Using a change of variables, we find
\[
u(t) =
\int_{t-\tau}^tT(t-r)P_s\vecd{f(r)}{0}dr+T(\tau)\int_{-\infty}^{t}T(t-r)P_s\vecd{f(r-\tau)}{0}dr.
\]
By \cref{lemma:ddeT}, the second term is in $\cZ$, we therefore focus on the first
term, which we denote by $u^1$. We fix $r \in \RSet$ and introduce
\[
  h(t,r) =
  \begin{cases}
    \pi_1T(t)P_s\vecd{f(r)}{0}, & \text{ if $t \geq 0$, } \\
    f_2(r)(t) = 0             , & \text{ if $t\in[-\tau,0)$.}
  \end{cases}
\]
We know from the previous lemma \cref{lemma:ddeT}
that 
\[
  \pi_2T(t)P_s\vecd{f(r)}{0}(\theta)=h(t+\theta,r) \qquad
  \textrm{for all $\theta\in[-\tau,0]$}\,.
\]
Owing to \cref{lemma:ddeT}, the only point left to show is $\pi_2(u^1(t))\in
W^{1,2}((-\tau,0),\RSet),\RSet)$. For all $\theta\in[-\tau,0]$ we have
\[
\pi_2(u(t))(\theta) = \int_{t-\tau}^{t+\theta}h(t-r+\theta,r)dr =
\int_{t-\tau}^{t+\theta} e^{v(t-r+\theta)}f(r)dr.
\]
Since $f\in C(\RSet,\RSet)$, it follows that $\pi_2(u^1(t))\in
W^{1,2}((-\tau,0),\RSet),\RSet)$ and $u(t)\in\cZ$.

Finally, we address the differentiability of the solution. For a fixed $t\in\RSet$
and $\varepsilon$, one finds
\[
\frac1\eps(u(t+\eps) - u(t)) = \frac1\eps(T(\eps)u(t)-u(t) )+
\frac1\eps\int_t^{t+\eps}T(t+\eps-r)P_s\vecd{f(r)}{0}dr.
\] 
By continuity of $f$, the last term converges to $P_s\vecd{f(r)}{0}$ in $\cX$ as
$\eps$ tends to zero. As $u(t)\in\cZ$ which is the domain of $L$,
\[
  \frac{T(\eps)u(t)-u(t)}{\eps} \to Lu(t) \qquad \textrm{as $\eps \to 0$.}
\]
Hence, $u$ is differentiable in $\cX$ and satisfies
\eqref{eq:linProbFastSubs} in $\cX$.
\end{proof}

\subsubsection{Slow-passage through Hopf bifurcation}
Looking for purely imaginary eigenvalues $\pm i\omega$ to the equation
$v-i\omega=e^{-i\omega\tau}$, one finds that this requires $|v|\leq 1$ by taking the
real part. By taking the norm, we find $\omega=\pm\sqrt{1-v^2}$. Then the Hopf
bifurcation point $v_H$ is solution of $v_H = \cos(\tau\sqrt{1-v_H^2})$.
We now consider such value of $v$ for which $\sigma_c=\{\pm i\omega\}$ with
$\omega>0$. We can henceforth proceed as in section \ref{subsec:FiniteDimResult} because the
normal form and the slow system is identical to the present one.

\section{Conclusions} \label{sec:conclusions}
In this paper, we have provided new local results on the outstanding problem of
proving the existence of canard solutions in infinite-dimensional
slow-fast dynamical systems as well as delayed bifurcation scenarios in this context. 
Namely, we have addressed the general case of systems
with $m$ slow variables and infinitely many fast variables, that is, systems for which
the fast component lives in a Banach space. In this general context, we have
proven center manifold reductions near points where normal hyperbolicity was lost as
these correspond to where canard dynamics can emerge. This effectively enabled us to
find local coordinates in which the original infinite-dimensional problem reduces to
an $m$-slow/$n$-fast system, where standard results from canard theory and delayed
bifurcations of ODEs apply. Therefore, we have obtained the existence of local canard
segments as well as slow passages through bifurcations in the general framework of
dynamical systems with $m$ slow and infinitely many fast variables. 
We then gave a complete rigorous description of such
results near a fold bifurcation of the original fast subsystem.
We also provided the main steps of the proofs in
other cases, like in the slow passage through a Turing bifurcation, which had not
been analysed before.
In the case of the slow passage through a Hopf bifurcation we could only show a
``short" delay since analyticity of the original problem is lost by the centre manifold
reduction. Hence, the proof of a ``long" delay is an open problem.  
Finally, we brought new results along similar lines in
slow-fast delay-differential equations. Every case that has been covered
theoretically was also accompanied by a computational example. Future work will
include connecting our local results to global ones, in order to be able to fully
describe canard solutions, in particular canard cycles, in infinite-dimensional
slow-fast dynamical systems, first with finitely many slow variables and then in
systems where both fast and slow variables are infinite dimensional, which is a
substantially more difficult problem. Yet the present work constitutes
a necessary rigorous initial first step towards achieving this research programme.

\section*{Acknowledgements}
We wish to thank Gr\'egory Faye for useful discussions regarding global orbits in
infinite-dimensional systems, and Anthony Roberts for sharing his survey on
centre-manifold reductions in infinite-dimensional systems.

\appendix 
\section{Slow-fast concepts for ODEs}\label{sec:slowFastODEs}
Here we recall basic notions about slow-fast (finite-dimensional) dynamical systems
and elements of geometric singular perturbation theory (GSPT)
\cite{Fenichel:1979,J95,GSPTphonebook,W20}.

\subsection{The standard GSPT setting}
For ease of reference, we recall the singular perturbation problem \eqref{eq:slowFastFin},
\begin{equation}\label{eq:slowFastFin-App}
\begin{aligned}
  \dot{u}& = F(u,v,\mu,\eps)\\
  \dot{v}& = \eps G(u,v,\mu,\eps),
\end{aligned}
\end{equation}
where $(u,v)\in\RSet^n\times\RSet^m$, $F$ and $G$ are smooth functions on
$\RSet^n\times\RSet^m\times\RSet^p\times \RSet$, and $0<\eps\ll 1$ is the timescale separation
parameter. The overdot ($\,\dot{}=d/dt$) refers to the fast timescale $t$ and, hence,
we refer to \eqref{eq:slowFastFin-App} as the fast system. Alternatively, we may
introduce the time transformation $d\tau=\eps\, dt$ which transforms
\eqref{eq:slowFastFin-App} to the equivalent {\em slow system}
\begin{equation}\label{eq:slowFastFinSlowTime}
\begin{aligned}
  \eps u'& = F(u,v,\mu,\eps)\\
      ~~v'& = G(u,v,\mu,\eps),
\end{aligned}
\end{equation}
where prime ($'=d/d\tau$) refers now to the slow timescale  $\tau$.
System \eqref{eq:slowFastFin-App} respectively \eqref{eq:slowFastFinSlowTime} are
topologically equivalent and solutions often consist of a mix of slow and fast
segments reflecting the dominance of one time scale or the other. 
As $\eps \to 0$, the trajectories of \eqref{eq:slowFastFin-App} converge during fast
segments to solutions of the $n$-dimensional \emph{layer problem}
\begin{equation}\label{eq:FastSubsys}
\begin{aligned}
  \dot{u}& = F(u,v,\mu,0)\\
  \dot{v}& = 0,
\end{aligned}
\end{equation}
while during slow segments, trajectories of~\eqref{eq:slowFastFinSlowTime} converge to solutions of
\begin{equation}\label{eq:SlowSubsys}
\begin{aligned}
          0& = F(u,v,\mu,0)\\
      ~~v'& = G(u,v,\mu,0),
\end{aligned}
\end{equation}
which is a $m$-dimensional differential-algebraic problem called  the \emph{reduced problem\/}. GSPT uses these lower-dimensional subsystems \eqref{eq:SlowSubsys} and \eqref{eq:FastSubsys} to predict the dynamics of the full $(n+m)$-dimensional system~\eqref{eq:slowFastFinSlowTime} or~\eqref{eq:slowFastFin-App} for $\eps > 0$.

\begin{definition}
The set 
\begin{equation}\label{def:S}
  S := \{(u, v) \in \RSet^n \times \RSet^m \;|\; F(u, v,\mu, 0) = 0 \}
\end{equation} 
is the set of equilibria of \eqref{eq:FastSubsys}. In general, this set $S$ defines a regular $m$-dimensional differentiable manifold, i.e. the Jacobian $D_{(u,v)}F$ evaluated along $S$ has full (row) rank $n$. We refer to $S$ as the {\em critical manifold}.
\end{definition}
\begin{remark}
In general, the set of singularities $S$ could be a union of disjoint manifolds or a union of manifolds intersecting along lower dimensional submanifolds. The theory exists for these cases as well.
\end{remark}
The existence of an $m$-dimensional manifold $S$ in \eqref{eq:FastSubsys} implies that its Jacobian evaluated along $S$ has at least $m$ zero eigenvalues corresponding to the $m$-dimensional tangent space $T_{(u,v)}S$ of $S$ at each point $(u,v)\in S$. We call these $m$ eigenvalues {\em trivial} associated with the $\dot{v}=0$ subsystem in \eqref{eq:FastSubsys}. The remaining $n$ eigenvalues are called {\em nontrivial} associated with the $n$-dimensional submatrix $D_uF|_S$.

\subsection{Normal hyperbolicity}
\begin{definition}
A subset $S_h \subseteq S$ is
called {\em normally hyperbolic} if all $(u,v) \in S_h$ are hyperbolic
equilibria of the layer problem \eqref{eq:FastSubsys}, i.e., the non-trivial eigenvalues have all nonzero real part. 
\begin{itemize}
\item We call a normally hyperbolic subset $S_a \subseteq S$ \emph{attracting\/} if all non-trivial eigenvalues have negative real parts for $(u,v)\in S_a$.
\item $S_r\subseteq S$ is called \emph{repelling\/} if all non-trivial eigenvalues have positive real parts for $(u,v)\in S_r$.
\item If $S_s \subseteq S$ is normally hyperbolic
and neither attracting nor repelling we say it is of \emph{saddle
type\/}.
\end{itemize}
\end{definition}
Normal hyperbolicity induces a (unique) splitting of the corresponding tangent space along $S$, i.e.~
\begin{equation}\label{def:split}
T_{(u,v)}\RSet^{n+m}= T_{(u,v)}S\oplus {\cal N}_{(u,v)}\,,\quad \forall (u,v)\in S_h\,,
\end{equation}
where $T_{(u,v)}S$ and ${\cal N}_{(u,v)}$ are invariant subspaces under the Jacobian 
$$
J=
\begin{pmatrix}
D_uF & D_vF\\
0 & 0
\end{pmatrix}
$$
of \eqref{eq:FastSubsys}. More precisely, $T_{(u,v)}S$ is in the kernel of $J$ and,
hence, corresponds to the $m$ trivial eigenvalues. This induces a linear map on the
$n$-dimensional quotient space ${\cal N}_{(u,v)}=T_{(u,v)}\RSet^{n+m}/T_{(u,v)}S$
corresponding to the $n$ non-trivial eigenvalues of $J$. The disjoint union of all
$T_{(u,v)}S$, $TS=\cup\, T_{(u,v)}S$, forms the {\em tangent bundle} along $S$, while
the disjoint union of all ${\cal N}_{(u,v)}$, ${\cal N}=\cup\, {\cal N}_{(u,v)}S$,
forms the corresponding {\em fast fibre bundle}.\\

The reduced problem~\eqref{eq:SlowSubsys} is a differential algebraic problem and
describes the evolution of the slow variables $v$ constrained to the critical
manifold $S$. As a consequence,  $S$ defines an interface between the two sub-systems
\eqref{eq:FastSubsys} and \eqref{eq:SlowSubsys}. 
%
%
In the case of normal hyperbolicity, the reduced vector field can be defined as a dynamical system
by appealing to the (unique) splitting \eqref{def:split} which defines a unique projection operator 
$$\Pi^{S_h}: T\RSet^{n+m}|_{S_h}= TS_h \oplus {\cal N} \to TS_h\,,$$ 
a map from $T\RSet^{n+m}$ onto the base space $TS_h$ along ${\cal N}$. This projection operator can be calculated explicitly, i.e. the reduced problem \eqref{eq:SlowSubsys} along $S$ is given by the following dynamical system
\begin{equation}
\begin{aligned}\label{eq:reduced-gen}
\begin{pmatrix}
\dot{u}\\
\dot{v}
\end{pmatrix}
=
\Pi^{S_h}
\begin{pmatrix}
F_1(u,v,\mu,0)\\
G(u,v,\mu,0)
\end{pmatrix}
& =
\begin{pmatrix}
0 & -(D_uF)^{-1}D_vF\\
0 & 1
\end{pmatrix}
\begin{pmatrix}
F_1(u,v,\mu,0)\\
G(u,v,\mu,0)
\end{pmatrix}
\\
& =
\begin{pmatrix}
-(D_uF)^{-1}D_vF G(u,v,\mu,0)\\
G(u,v,\mu,0)
\end{pmatrix}
\,,
\end{aligned}
\end{equation}
where $F_1(u,v,\mu,0)$ is the $O(\eps)$ term in the power series expansion of the first component $F(u,v,\mu,\eps)$ with respect to  $\eps$ of the vector field \eqref{eq:slowFastFin-App}; $G(u,v,\mu,0)$ is the corresponding other (leading order) slow component. Note that $D_uF$ is a regular square matrix due to normal hyperbolicity. This regularity implies via the {\em implicit function theorem} that the critical manifold has a graph representation $S=\{(u,v)\in\RSet^{n+m}: u=h(v)\}$.
Hence, in the case of normal hyperbolicity, the reduced problem \eqref{eq:SlowSubsys} can be studied in this coordinate chart $v$ given by
\begin{equation}\label{reduced:NH}
\dot{v} = G(h(v),v,\mu,0)\,.
\end{equation}

The notion  of normal hyperbolicity  is central to the study of systems of the
form~\cref{eq:slowFastFin-App} as key assumption for the persistence of segments of
$S$, which are invariant for $\eps=0$, as local invariant \textit{slow manifold}
$S_\eps$ for small enough $\eps>0$. These major results have been obtained in the
mid-1970s by Neil Fenichel~\cite{Fenichel:1979}.

\subsection{Loss of normal hyperbolicity}
Geometrically, loss of normal hyperbolicity occurs (generically) along
codimension-one submanifold(s) of $S$ where nontrivial eigenvalue(s) of the layer
problem crosses the imaginary axis. Within this set $S\backslash S_h$, we distinguish
two subsets ${\cal AH}$ and ${\cal F}$.
\begin{definition}
Consider the two cases of loss of normal hyperbolicity associated with the two
generic codimension-one bifurcations in the layer problem \eqref{eq:FastSubsys}:
\begin{itemize}
\item  the set ${\cal AH}$ of Andronov-Hopf points associated with the crossing of a
  pair of complex conjugate eigenvalues (with nonzero imaginary part),
\item the set ${\cal F}$ of folds or saddle-node points associated with the crossing
  of a real eigenvalue\,.
\end{itemize}
\end{definition}
The layer problem is thus considered a bifurcation problem; see, e.g.,
\cite{kuznetsov2013elements}. The main question is what happens to this bifurcation
structure as the singular bifurcation parameter $0<\eps\ll 1$ is turned on, i.e.~when
the `bifurcation parameter' $v$ starts to evolve slowly. To answer this question, we
need to understand the reduced problem in a neighbourhood of these codimension-one
subsets.

In the case ${\cal AH}$, we note that $D_uF$ is still a regular square matrix, i.e.
$\det (D_uF) \neq 0$ everywhere along $S$ including ${\cal AH}$. Hence the reduced
problem can be still studied in the slow coordinate chart given by
\eqref{reduced:NH}. In general, the reduced vector field obeys $G(h(v),v,\mu,0)\neq
0$ along ${\cal AH}\subset S$. Thus the reduced flow crosses from one normally
hyperbolic branch of $S$ via ${\cal AH}\subset S$ to another normally hyperbolic
branch of $S$. Assuming this reduced flow is from an attracting branch $S_a$ to a
repelling branch $S_{r/s}$ this leads to the phenomenon of delayed loss of stability
through an Andronov-Hopf bifurcation; see \cite{neishtadt87,neishtadt88,HKSW16} for
details. 

In the case ${\cal F}$, the matrix $D_uF$ is singular, i.e. $\det (D_uF) = 0$ along
${\cal F}$. So, we cannot use \eqref{reduced:NH} to describe the reduced flow. While
system \eqref{eq:reduced-gen} is also not well defined near fold(s) ${\cal F}\subset
S$, we can make an equivalence transformation as follows: First, we use the identity
$(D_uF)^{-1}= \mathrm{adj} (D_uF)/(\det D_uF)$ where $\mathrm{adj} (D_uF)$ denotes
the adjoint (or adjugate) of the matrix $D_uF$, i.e.~the transpose of the co-factor
matrix, which gives
\begin{equation}
\begin{aligned}\label{eq:reduced-gen-2}
\begin{pmatrix}
\dot{u}\\
\dot{v}
\end{pmatrix}
& =
\begin{pmatrix}
-\dfrac{1}{\det (D_uF)}\,\mathrm{adj} (D_uF)D_vF G(u,v,\mu,0)\\
G(u,v,\mu,0)
\end{pmatrix}
\,.
\end{aligned}
\end{equation}
This isolates the singularity into the scaler function $\det (D_uF)$. We note that
the matrix $\mathrm{adj} (D_uF)$ is still well defined and of rank one, when $D_uF$
has rank deficiency of one (along ${\cal F}$). Now, we can remove the scalar
singularity (pole)  by the time transformation
$$
d\tau= - \det (D_uF) d\tau_1
$$
which gives the corresponding desingularised system
\begin{equation}
\begin{aligned}\label{eq:desing}
\begin{pmatrix}
\dot{u}\\
\dot{v}
\end{pmatrix}
& =
\begin{pmatrix}
\mathrm{adj} (D_uF)D_vF G(u,v,\mu,0)\\
- \det (D_uF)G(u,v,\mu,0)
\end{pmatrix}
\,,
\end{aligned}
\end{equation}
where with a slight abuse of notation the overdot (\,$\dot{}=d/d\tau_1$) refers now
to the new slow time $\tau_1$. Importantly, \eqref{eq:desing} represents a dynamical
system without singularities that can be analysed by standard dynamical systems
tools even in a neighbourhhood of ${\cal F}\subset S$. Note that systems
\eqref{eq:reduced-gen-2} and \eqref{eq:desing} are equivalent on normally hyperbolic
branches where $\det (D_uF)>0$, and these systems are also equivalent up to an
orientation change on normally hyperbolic branches where $\det (D_uF)<0$. 
\begin{definition}
Any point  $(u,v)\in {\cal F}$ where  
\begin{equation}\label{def:reg-jump}
 \mathrm{adj} (D_uF)D_vF G(u,v,\mu,0) \neq 0
 \end{equation}
is called a {\em regular jump point}. 
\end{definition}
Note, solutions of the reduced problem \eqref{eq:reduced-gen-2} in a neighbourhood of
${\cal F}$ approach regular jump points \eqref{def:reg-jump} in forward or backward
time and they cease to exist at regular jump points due to a finite time blow-up,
i.e., the reduced problem \eqref{eq:reduced-gen-2} has a pole at jump points. This
finite time blow-up of the reduced flow at a jump point  happens exactly in the
eigendirection of the defect of the layer problem which defines the locus (and
direction) where a switch from slow to fast dynamics (or vice versa) in the full
system is possible. This, together with an adequate global return mechanism via the
layer problem may provide the seed for a singular relaxation cycle, a concatenation
of slow and fast orbit segments that form a loop. For persistence results of such
relaxation cycles under small perturbations $\eps\ll 1$ we refer the reader to, e.g.,
\cite{KS01b,SW04}.

\begin{definition}
Any point  $(u,v)\in {\cal F}$ where  
\begin{equation}\label{def:fold-sing}
 \mathrm{adj} (D_uF)D_vF G(u,v,\mu,0) = 0
 \end{equation}
is called a {\em folded singularity}. 
\end{definition}
Folded singularities are equilibria of the desingularised problem \eqref{eq:desing}
but not necessarily of the corresponding reduced problem \eqref{eq:reduced-gen-2}
itself. At such a folded singularity, the vector field of \eqref{eq:reduced-gen-2}
contains indeterminate forms which could prevent a finite time blow-up of certain
solutions  of the reduced problem approaching a folded singularity. Such special
solutions of the reduced problem that are able to pass in finite time through such a
folded singularity from one branch of $S$ to another are called {\em canards}, and
they play an important role in understanding the genesis of relaxation oscillations
as well as in creating more complex oscillatory patterns in singular perturbation
problems. The terminology `canard' was introduced in \cite{BCDD1981}, and
we refer the interest reader to important persistence results in, e.g.,
\cite{DR96,KS01b,Ben83,SW01,W05,W12}.

\bibliography{references}
\bibliographystyle{siamplain}
 
\end{document}